\documentclass[a4paper,draft]{amsart}
\usepackage[utf8]{inputenc}
\usepackage{amsmath}
\usepackage{amssymb}
\usepackage{amscd}
\usepackage{amsthm}
\usepackage[centering]{geometry}
\usepackage{amsxtra}
\usepackage{pifont}
\usepackage{mathabx}
\usepackage{xcolor}

\definecolor{darkred}{rgb}{0.5,0,0}
\definecolor{darkgreen}{rgb}{0,0.5,0}
\definecolor{darkblue}{rgb}{0,0,0.5}

\usepackage{microtype}
\usepackage[linktocpage,pdftex]{hyperref}
\hypersetup{linktocpage
           ,pdfborder={0 0 0}
           ,colorlinks
           ,linkcolor=darkblue
           ,filecolor=darkgreen
           ,urlcolor=darkred
           ,citecolor=darkblue
           }

\numberwithin{equation}{section}

\newcommand{\Lep}{L}

\newcommand{\Z}{{\mathbb Z}}

\newcommand{\N}{\mathbb N}

\newcommand{\cK}{{\mathcal K}}
\newcommand{\cL}{{\mathcal L}}
\newcommand{\cA}{{\mathcal A}}
\newcommand{\cF}{{\mathcal F}}

\newcommand{\cd}{{\mathcal D}}
\newcommand{\mcard}{\mathrm{card}}

\RequirePackage{ifthen}

\providecommand{\norm}[2][]{\lVert#2\rVert\ifthenelse{\equal{}{#1}}{}{_{#1}}}
\providecommand{\bignorm}[2][]{\bigl\lVert#2\bigr\rVert\ifthenelse{\equal{}{#1}}{}{_{#1}}}
\providecommand{\Bignorm}[2][]{\Bigl\lVert#2\Bigr\rVert\ifthenelse{\equal{}{#1}}{}{_{#1}}}
\providecommand{\biggnorm}[2][]{\biggl\lVert#2\biggr\rVert\ifthenelse{\equal{}{#1}}{}{_{#1}}}
\providecommand{\Biggnorm}[2][]{\Biggl\lVert#2\Biggr\rVert\ifthenelse{\equal{}{#1}}{}{_{#1}}}
\providecommand{\spr}[3][]{\langle#2,#3\rangle\ifthenelse{\equal{}{#1}}{}{_{#1}}}
\providecommand{\bigspr}[3][]{\bigl\langle#2,#3\bigr\rangle\ifthenelse{\equal{}{#1}}{}{_{#1}}}
\providecommand{\Bigspr}[3][]{\Bigl\langle#2,#3\Bigr\rangle\ifthenelse{\equal{}{#1}}{}{_{#1}}}
\providecommand{\biggspr}[3][]{\biggl\langle#2,#3\biggr\rangle\ifthenelse{\equal{}{#1}}{}{_{#1}}}
\providecommand{\Biggspr}[3][]{\Biggl\langle#2,#3\Biggr\rangle\ifthenelse{\equal{}{#1}}{}{_{#1}}}
\let\originald\d 
\renewcommand{\d}{\ifthenelse{\boolean{mmode}}{\mathrm d}{\originald}}

\let\originali\i 
\renewcommand{\i}{\ifthenelse{\boolean{mmode}}{\mathrm i}{\originali}}

\newtheorem{theorem}{Theorem}[section]
\newtheorem{lemma}[theorem]{Lemma}
\newtheorem{prop}[theorem]{Proposition}
\newtheorem{cor}[theorem]{Corollary}

\newtheorem{defi}[theorem]{Definition}
\newtheorem{example}[theorem]{Example}

\newtheoremstyle{rremark}%
       {1.8ex\@plus1ex}     
       {2.1ex\@plus1ex\@minus.5ex} 
       {\normalfont}        
       {0pt}                
       {\bfseries}          
       {.}                  
       {.5em}               
       {}                   

\theoremstyle{rremark}
\newtheorem{remark}[theorem]{Remark}

\begin{document}
\title{Leptin densities in amenable groups}

\author{Felix Pogorzelski}
\address{Institut f\"ur Mathematik, Universit\"at Leipzig, Augustusplatz 10, 04109 Leipzig, Germany}
\email{felix.pogorzelski@math.uni-leipzig.de}

\author{Christoph Richard}
\address{Department f\"{u}r Mathematik, Friedrich-Alexander-Universit\"{a}t Erlangen-N\"{u}rnberg,
Cauerstrasse 11, 91058 Erlangen, Germany}
\email{christoph.richard@fau.de}

\author{Nicolae Strungaru}
\address{Department of Mathematical Sciences, MacEwan University \\
10700--104 Avenue, Edmonton, AB, T5J 4S2\\
and 
Institute of Mathematics ``Simon Stoilow''\\
Bucharest, Romania}
\email{strungarun@macewan.ca}

\begin{abstract}
Consider a positive Borel measure on a locally compact group. We define a notion of uniform density for such a measure, which is based on a group invariant introduced by Leptin in 1966. We then restrict to unimodular amenable groups and to translation bounded measures. In that case our density notion coincides with the well-known Beurling density from Fourier analysis, also known as Banach density from dynamical systems theory.
We use Leptin densities for a geometric proof of the model set density formula, which expresses the density of a uniform regular model set in terms of the volume of its window, and for a proof of uniform mean almost periodicity of such model sets.
\end{abstract}

\maketitle

\section{Introduction}

This article addresses asymptotic frequencies of point sets in locally compact groups $G$ that are amenable \cite{P88}.\footnote{We will implicitly assume that such groups always satisfy the Hausdorff property.} We introduce a canonical notion of uniform density, which is intimately related to the group invariant $I(G)$ introduced by Horst Leptin in \cite{L66}, and we thus call it the \textit{Leptin density} of a locally finite point set or more generally of a positive Borel measure on $G$. 
Let us restrict to amenable locally compact groups that are unimodular. This setting encompasses locally compact groups that are exponentially bounded \cite[Prop.~6.8]{P88}, such as locally compact abelian groups. Let us further restrict to translation bounded positive Borel measures, as defined below. In that case, a main result of our article asserts that our notion of Leptin density coincides with the classical notion of Banach density, also known as Beurling density or Beurling-Landau density. The latter densities may be evaluated on any so-called strong F{\o}lner or van Hove net. We will show that such nets can be obtained from any F{\o}lner net by a simple thickening procedure.

Our result might be useful for sampling and interpolation problems on such groups, as necessary conditions for sampling and interpolation  are usually formulated in terms of the Beurling density of the underlying point set. Our definition of Leptin density is in fact inspired by work of Gröchenig, Kutyniok and Seip \cite{GKS08} in that direction, compare \cite[Lem.~9.3]{RiSchu20}. In Euclidean space, it was already noted by Garbardo \cite{G18} that the Beurling density can be defined without referring to boxes. Due to its connection to amenability, our observation clarifies why the Beurling density is independent of the chosen averaging sets. Making proper use of Leptin densities might streamline existing proofs, as one does not have to resort to particular averaging sets. The notion of Leptin density might also provide insight as how to extend the notion of Beurling density beyond the group setting, compare \cite{FGHKR17, MR21}.

\smallskip

A further focus of this article is on applications to mathematical diffraction theory \cite{Hof1, BM, RS17}. In that field, it is common to work with van Hove (strong F{\o}lner) sequences, which restricts the analysis to $\sigma$-compact amenable groups. However one might drop the assumption of $\sigma$-compactness and use van Hove (strong F{\o}lner) nets, at least if one does not use dynamical arguments such as the pointwise ergodic theorem. In fact one might avoid nets at all, using Leptin densities. As a prominent example, consider the class of model sets in $\sigma$-compact locally compact abelian groups. Such point sets are obtained from a cut-and-project construction, originating in Meyer's ground breaking work during the 70ies \cite{Mey1, Mey2, Mey3, Schreiber71, Schreiber73}. Further investigation had been advocated by Moody in the 90ies \cite{RVM3, Moody00}, motivated by  the experimental discovery of quasicrystals. In fact the cut-and-project construction had been re-developed independently in that context, see e.g.~the historical discussion in \cite[Sec.~2]{RS17b}. By now the class of model sets constitutes the most important example of pure point diffractive structures, see e.g.~the monograph \cite{BG2} and references therein. Moreover, the last couple of years has seen an emerging theory of aperiodic order beyond the abelian situation, see  \cite{BHP17, BHP17II, BHP17III} for regular model sets in general locally compact second countable groups and homogeneous spaces, as well as \cite{BH, Mac, BHII}  for foundations of a new theory on approximate lattices.

\smallskip

The so-called density formula expresses the density of a regular model set in terms of the volume of its window.   A first proof for Euclidean space, which relies on the Poisson summation formula,  goes back to Meyer \cite{Mey2, Mey3}, compare \cite{RS17} for a recent account. In varying degrees of generality, alternative proofs have been given by geometric methods \cite{sch98}, dynamical methods \cite{Mo02}, and by methods based on almost periodicity, see e.g.~ \cite{BM,G05}.
We refer to \cite[Sec.~3]{HR15} for a detailed account of the history of the density formula until 2015. For non-abelian regular model sets as defined in \cite{BHP17}, a variant of density formula can be deduced using dynamical methods, see Remark~\ref{rem:densityBHP} below. This subsumes results for the non-abelian setting from \cite[Rem.~3.3]{HR15} and \cite[Sec.~9.3]{KR}.  In the present paper we use Leptin densities in order to give a geometric proof of the density formula for uniform regular model sets in amenable groups.
The method is via tilings with disjoint copies of translates of a relatively compact fundamental domain arising from a uniform lattice. In Euclidean space, a similar combinatorial proof was given by Schlottmann in \cite{sch98}. We emphasize at this point that the lattice in the cut-and-project scheme being cocompact is a crucial ingredient in the line of argumentation. In particular, the density formula presented here is not the most general version.  We will also use Leptin densities in order to analyse almost periodicity of a regular model set. This complements previous results in the $\sigma$-compact abelian case \cite{BM, LR07, GM14, G18, LSS20}.

\smallskip

The article is structured as follows. In the following section, we fix the class of point sets, using a measure-theoretic description. Our setting is that of right-translation bounded measures.
Section 3 is devoted to the notion of Leptin density and discusses some of its properties. In Section 4 we prove that Leptin densities coincide with Beurling densities, when defined without resorting to an averaging net.
In Section 5, we discuss various notions of boundary and how they characterise averaging nets. Moreover, we describe Beurling and Leptin densities via approximation by certain F{\o}lner nets.
In Section 6, we compute the Leptin density of a cocompact lattice in a locally compact amenable group via a tiling argument for fundamental domains. The same strategy is used in Section~7, where we give a geometric proof of the density formula for (regular) model sets in amenable groups. In Section~8, we discuss uniform versions of almost periodicity with focus on a description avoiding nets. We prove uniform mean almost periodicity of regular model sets using the notion of Leptin density. For the convenience of the reader, Appendix~A collects basic facts about nets, Appendix~B explains box decomposition arguments that are useful for monotilable groups, instead of our more general approach. Appendix C describes the relation between Leptin densities and the density notion in \cite{GKS08}, which has inspired our work.

\subsection*{Acknowledgements} CR would like to acknowledge visiting funds from the Universit\"at Leipzig in 2019 and from MacEwan University Edmonton in 2020. NS was supported by the Natural Sciences and Engineering Council of Canada (NSERC), via grant 2020-00038, and he would like to thank for the support.
FP and CR are grateful to the organizers of the conference ``Model sets and aperiodic order'' (Durham UK, September 3rd -- 7th, 2018), which provided an excellent  environment to draft some of the ideas that have now been developed in this paper. CR would like to thank Hartmut Führ and Karlheinz Gröchenig for discussions, and Karlheinz Gröchenig for funding an inspiring visit to NuHAG at Vienna University in October 2021, where part of this work has been presented.

\section{Delone measures}

\subsection{Definition of Delone measures}\label{sec:DM}

Instead of counting the number of points in locally finite sets, we use the more general setting of (positive) Borel measures. For some of our results we restrict to translation bounded measures, a notion which encompasses uniformly discrete point sets. See the discussion below.
It is somewhat common in works on non-abelian aperiodic order to define uniform discreteness and relative denseness for point sets via left translations, cf.\@ e.g.\@ \cite{mr13, HR15, BHP17, BHP17II, BH, BP}. Deviating from this convention we will work with the (of course equivalent) setting of uniform discreteness and relative denseness from the right. In the presence of a metric this means describing the topology by means of a right-invariant instead of a left-invariant metric. The reason for this choice is the interplay between uniform counting of points along asymptotically invariant  sets. For amenable groups, counting quantities defined via approximation by left-asymptotically invariant (left-F{\o}lner) sets show desirable uniformity properties with respect to translations from the right. This phenomenon has also been observed in the theory of aperiodic order, see e.g.\@ \cite{mr13, PS16, BHP}. In the present paper, we decide to work with left-asymptotic invariance, which is also standard in the literature, and we consequently accept the topological structure to be from the right.

\medskip

Throughout the paper, we denote by $G$ a locally compact topological group. When we speak of a locally compact group $G$ in this work, we will implicitly assume that $G$ also satisfies the Hausdorff property, i.e.,\@ any two points in $G$ can be separated by disjoint open neighborhoods.
For simplicity we will call a positive Borel measure a measure in the following.

\begin{defi}[Delone measure]
Let  $\nu$ be a measure on a locally compact group $G$.

\begin{itemize}
\item[(a)] We call $\nu$ upper translation bounded if there exist $C_u<\infty$ and a compact symmetric unit neighborhood $B_u\subseteq G$ such that $\sup_{x\in G} \nu(B_ux)\le C_u$.

\item[(b)] We call $\nu$ lower translation bounded if there exist $C_l>0$ and a compact symmetric unit neighborhood $B_l\subseteq G$ such that $\inf_{x\in G} \nu(B_lx)\ge C_l$.

\item[(c)] We call $\nu$ a Delone measure if $\nu$ it is both upper and lower translation bounded.
\end{itemize}
\end{defi}

\begin{remark}\label{rem:utb}
In the above definition, by convenience we restricted to compact symmetric unit neighborhoods. Using covering arguments, it is readily seen that $\nu$ is upper translation bounded iff $\sup_{x\in G}\nu(Kx)<\infty$ for every compact $K\subseteq G$. In particular, upper translation boundedness implies local finiteness, i.e., $\nu(A)$ is finite for every compact $A$. For a complex measure, upper and lower translation boundedness can be defined via its (positive) variation measure. Upper translation boundedness is traditionally called translation boundedness \cite{ARMA1}. The name Delone measure is motivated from weak Delone sets \cite{LW03}, i.e., from point sets that are weakly uniformly discrete and relatively dense. For point sets $\Lambda$, weak uniform discreteness is equivalent to upper translation boundedness of the associated point measure defined via $\nu(A)=\mathrm{card}(\Lambda\cap A)$. Relative denseness is equivalent to lower translation boundedness of the associated point measure. For point sets, uniform discreteness is also called separatedness, and weak uniform discreteness is also called relative separatedness, see e.g.~\cite{FGHKR17, MR21}.
\end{remark}

In this article, we will analyse various notions of density of a measure. Upper translation boundedness will result in finite upper density.  Lower translation boundedness will result in positive lower density. 
Examples of Delone measures are any (left or right) Haar measure on a unimodular locally compact group $G$ and, for a uniform lattice $L$ in $G$, the Haar measure on $L$ viewed as a measure on $G$.

\smallskip

\subsection{Standard estimates on unimodular groups}\label{subsec:se}

For later use, we provide standard estimates for Delone measures on unimodular groups. These are based on the following observation from \cite{OW87}.  In the sequel, $AB$ denotes the Minkowski product of $A,B\subseteq G$, i.e., we have $AB=\{ab: a\in A, b\in B\}$.

\begin{lemma}\label{lem:nOW}
Let $G$ be a unimodular locally compact group with Haar measure $m$.
Consider any compact $A\in\mathcal K$ and any unit neighborhood $B$. Take $\{a_1,\ldots, a_n\}\subseteq A$ maximal such that $Ba_i\cap Ba_j=\varnothing$ for $i\ne j$. Then
\begin{displaymath}
\bigcup_{i=1}^n B a_i \subseteq BA \ , \qquad
A \subseteq \bigcup_{i=1}^n B^{-1}Ba_i \ .
\end{displaymath}
where the union on the lhs is disjoint.
In particular we have
\begin{displaymath}
n\cdot m(B) \le m(BA)\ , \qquad m(A) \le n\cdot m(B^{-1}B) \ .
\end{displaymath}
\end{lemma}

\begin{proof}
The first inclusion is obvious. For the second inclusion, assume $a\in A$ satisfies $a\notin \bigcup_{i=1}^n B^{-1}Ba_i$. Then for arbitrary fixed $i\in \{1,\ldots, n\}$ we have $a\notin B^{-1} B a_i$. Hence $ba \ne b'a_i$ for all $b,b'\in B$, which implies $Ba\cap Ba_i=\varnothing$. As $i\in \{1,\ldots, n\}$ was arbitrary, this contradicts maximality. The inequalities for the Haar measure follow from unimodularity.
\end{proof}

We denote by $\mathcal K=\mathcal K(G)$ the collection of nonempty compact subsets of $G$.
For an upper translation bounded measure we have the following upper bound.

\begin{lemma}\label{lem:tbb}
Let $G$ be a unimodular locally compact group with Haar measure $m$.
Let $\nu$ be an upper translation bounded measure on $G$. Take a compact symmetric unit neighborhood $B_u\subseteq G$ and a finite number $C_u$ such that $\nu(B_u^2x)\le C_u$ for all $x\in G$. We then have for any $A\in \mathcal K$ the estimate
\begin{displaymath}
\nu(A)\le \frac{C_u}{m(B_u)} \cdot m(B_uA) \ .
\end{displaymath}
\end{lemma}

\begin{proof}
Fix any $A\in\mathcal K$ and consider the setting of Lemma~\ref{lem:nOW} with $B=B_u$. We then have
\begin{displaymath}
\nu(A) \le \nu\left(\bigcup_{i=1}^n B_u^2a_i \right) \le n\cdot \max_{1 \leq i \leq n}\nu(B_u^2a_i)  \le \frac{C_u}{m(B_u)} \cdot m(B_uA) \ .
\end{displaymath}
\end{proof}

For a measure that is lower translation bounded, we have the following lower bound.

\begin{lemma}\label{lem:tba}
Let $G$ be a unimodular locally compact group with Haar measure $m$.
Let $\nu$ be a lower translation bounded measure on $G$. Take a compact symmetric unit neighborhood $B_l\subset G$ and a positive number $C_l$ such that $\nu(B_lx)\ge C_l$ for all $x\in G$. We then have for any compact set $A\in \mathcal K$ the estimate
\begin{displaymath}
\nu(B_l A)\ge \frac{C_l}{m(B_l^2)} \cdot m(A)\ .
\end{displaymath}
\end{lemma}

\begin{proof}
Fix any $A\in\mathcal K$ and consider the setting of Lemma~\ref{lem:nOW} with $B=B_l$. We then have
\begin{displaymath}
\nu(B_lA) \ge \nu\left(\bigcup_{i=1}^n B_la_i\right) =\sum_{i=1}^n \nu(B_la_i)\ge n\cdot C_l \ge \frac{C_l}{m(B_l^2)}\cdot m(A) \ .
\end{displaymath}
\end{proof}

\section{Leptin density of a measure}

We will define a certain uniform density that is based on the group invariant $I(G)$ introduced by Leptin \cite{L66}, see also \cite[Cor.~4.14]{P88}. Our notion is inspired by the work of Gr\"ochenig, Kutyniok and Seip in \cite{GKS08}, see Appendix C for a discussion of the connection.

\subsection{Definition}
Let $G$ be a locally compact group. Take any left Haar measure $m$ on $G$. Let $\mathcal K$ denote the collection of nonempty compact subsets of $G$, and let $\mathcal K_p\subseteq \mathcal K$ denote the collection of
compact subsets of $G$ of positive left Haar measure. Consider
\begin{displaymath}
I(G)=\sup_{K\in \mathcal K}\inf_{A\in \mathcal K_p} \frac{m(KA)}{m(A)} \ .
\end{displaymath}
Note that in the definition of $I(G)$ we may restrict to $K$ being a compact symmetric unit neighborhood without loss of generality.
The quantity $I(G)$ is a  group invariant as it does not depend on the choice of the left Haar measure.
It exhibits the following dichotomy.
\begin{lemma}\label{ref:lemcharamen}
Let $G$ be a locally compact group with left Haar measure $m$. Then either $I(G)=1$ or $I(G)=\infty$.
\end{lemma}

\begin{proof}
Note that $I(G)\ge1$ by definition.
For any compact unit neighborhood $K\in \mathcal K_p$ consider
\begin{displaymath}
C(K)=\inf_{A\in \mathcal K_p} \frac{m(KA)}{m(A)}
\end{displaymath}
and note that $C(K^2)\ge C(K)^2$. Assume now that $I(G)>1$. Then there exist a compact unit neighborhood $K_0\in \mathcal K_p$ and $\varepsilon_0>0$ such that $C(K_0)>1+\varepsilon_0>1$. This implies $I(G)\ge C(K_0^n)\ge C(K_0)^n\to\infty$. Thus $I(G)=\infty$.
\end{proof}

Recall that a locally compact group $G$ is {\em amenable} if it admits a left-invariant mean on $L^{\infty}(G)$. The following important result is well known, see e.g. \cite{Eme68, Gre73} and the monograph \cite{P88}.
\begin{theorem}\cite[Cor.~4.14]{P88}\label{theo:am}
Let $G$ be a locally compact group. Then $G$ is amenable if and only if $I(G)=1$. \qed
\end{theorem}
We now define the notion of Leptin density of a measure. Although our definition is given in a general setting, the notion of Leptin density seems most useful in amenable unimodular locally compact groups. Moreover it will be infinite on measures that are not upper translation bounded. This will be discussed below.

\begin{defi}[Leptin densities]\label{def:lep}
Let $G$ be a locally compact group. Let $m$ be a left Haar measure on $G$. For a measure $\nu$ on $G$ consider
\begin{displaymath}
\Lep_\nu^-=\sup_{K\in \mathcal K}\inf_{A\in \mathcal K_p} \frac{\nu(KA)}{m(A)} \ , \qquad
\Lep_\nu^+=\inf_{K\in \mathcal K}\sup_{A\in \mathcal K_p} \frac{\nu(A)}{m(KA)}  \ .
\end{displaymath}
We call $\Lep_{\nu}^{-}$ and $\Lep_{\nu}^{+}$ the lower Leptin density respectively the upper Leptin density of the measure $\nu$. If $\Lep_{\nu}^{-}=\Lep_{\nu}^{+}=\Lep_\nu$, we call $\Lep_\nu$ the Leptin density of $\nu$.
\end{defi}

\begin{remark}\label{rem:Lepmot}
In the above definition we may assume $e\in K$ without loss of generality, due to left invariance of the Haar measure.
The attribution to Leptin is motivated by $\Lep^-_m=I(G)$ and $\Lep^+_m=I(G)^{-1}$. Note that  the left Haar measure $m$ has Leptin density $1$ if and only if $G$ is amenable, as a consequence of Theorem~\ref{theo:am}. 
\end{remark}

\subsection{Some properties of Leptin densities}

Let us analyse how Leptin densities behave with respect to translation. For a given measure $\nu$ on $G$ we consider its left translation $\nu\mapsto \delta_t*\nu$ and its right translation $\nu\mapsto \nu*\delta_t$  for $t\in G$, where $(\delta_t*\nu)(A)=\nu(t^{-1}A)$ and $(\nu*\delta_t)(A)=\nu(At^{-1})$.
The following result is obvious.
\begin{lemma}\label{sec:lepinv}
Let $\nu$ be a measure on a locally compact group $G$. Then the Leptin densities of $\nu$ are invariant under left translation, i.e., we have for all $t\in G$ that
\begin{displaymath}
\Lep^-_{\delta_t*\nu}=\Lep^-_\nu \ , \qquad \Lep^+_{\delta_t*\nu}=\Lep^+_\nu \ .
\end{displaymath}
If $G$ is assumed to be unimodular, then the Leptin densities of $\nu$ are also invariant under right translation, i.e., we have for all $t\in G$ that
\begin{displaymath}
\Lep^-_{\nu*\delta_t}=\Lep^-_\nu \ , \qquad \Lep^+_{\nu*\delta_t}=\Lep^+_\nu \ .
\end{displaymath}
\qed
\end{lemma}

Leptin densities reflect amenability of the underlying group. Whereas they are degenerate for Delone measures in non-amenable groups, they are well behaved in the amenable case.
\begin{lemma}\label{lem:Smp}
Let $\nu$ be a measure on a locally compact group $G$. If $G$ is amenable, then $\Lep^-_\nu\le \Lep^+_\nu$.
\end{lemma}

\begin{proof}
Assume $\Lep^+_\nu<\Lep^-_\nu$ and pick $a<b$ such that
$\Lep^+_\nu< a <b <\Lep^-_\nu$.
Since $b<\Lep^-_\nu$, there exists $K_b\in \cK$ such that for all $A\in\cK_p$ we have
$b \cdot m(A)\le \nu(K_bA)$. Since $a > \Lep^+_\nu$, there exists $K_a\in\cK$ such that for all $A\in\cK_p$ we have
$a \cdot m(K_aA)\ge \nu(A)$. Now fix arbitrary $A\in\cK_p$. We thus have
$b \cdot m(A) \le \nu(K_bA)\leq a \cdot m(K_aK_bA)$.
Therefore, for all $A\in\cK_p$ we have
\begin{displaymath}
\frac{ m(K_aK_bA)}{m(A)} \geq \frac{b}{a} >1 \,.
\end{displaymath}
Using the notation of the proof of Lemma~\ref{ref:lemcharamen}, this implies $C(K_aK_b)>1$. Hence $I(G)=\infty$, and  $G$ is not amenable by Theorem~\ref{theo:am}.
\end{proof}

For a converse of the above statement in the unimodular case, see Lemma~\ref{lem:elprop}.

\subsection{Leptin densities of Delone measures in unimodular groups}

In this subsection, we restrict to unimodular locally compact groups. The standard estimates of Section~\ref{subsec:se} have the following consequence.

\begin{lemma}[Standard estimates]\label{lem:seLepDel}
Let $G$ be a unimodular locally compact group. Let $\nu$ be any Delone measure on $G$. Take compact symmetric unit neighborhoods $B_l, B_u\subseteq G$ and positive finite constants $C_l, C_u$ such that
\begin{displaymath}
C_l\le  \nu(B_lx) \ , \qquad \nu(B_u^2x) \le C_u
\end{displaymath}
for all $x\in G$. Then the following  estimates hold.
\begin{displaymath}
\begin{split}
 \frac{C_l}{m(B_l^{2})}\cdot I(G)&\le \Lep_\nu^-\le \frac{C_u}{m(B_u)}\cdot I(G) \ , \\
\frac{C_l}{m(B_l^{2})}\cdot I(G)^{-1} &\le \Lep_\nu^+\le \frac{C_u}{m(B_u)} \cdot I(G)^{-1}  \ .
\end{split}
\end{displaymath}
Here the upper estimates rely on upper translation boundedness, and the lower estimates rely on lower translation boundedness.
\end{lemma}

\begin{proof}
We consider $\Lep_\nu^-$ first and use the  estimate in Lemma~\ref{lem:tba} to obtain
\begin{displaymath}
\Lep_\nu^-\ge \sup_{K\in \mathcal K}\inf_{A\in \mathcal K_p} \frac{\nu(B_lKA)}{m(A)}
\ge \frac{C_l}{m(B_l^{2})}\cdot \sup_{K\in \mathcal K}\inf_{A\in \mathcal K_p} \frac{m(KA)}{m(A)}
=  \frac{C_l}{m(B_l^{2})}\cdot I(G) \ .
\end{displaymath}
For an upper estimate, we use the  estimate in Lemma~\ref{lem:tbb} to obtain
\begin{displaymath}
\Lep_\nu^-=\sup_{K\in \mathcal K}\inf_{A\in \mathcal K_p} \frac{\nu(KA)}{m(A)}
\le \frac{C_u}{m(B_u)}\cdot \sup_{K\in \mathcal K}\inf_{A\in \mathcal K_p} \frac{m(B_uKA)}{m(A)}
\le  \frac{C_u}{m(B_u)}\cdot I(G) \ .
\end{displaymath}
The proofs of the assertions on $\Lep_\nu^+$ are analogous and use Lemma~\ref{lem:tbb}.
\end{proof}
The following lemma relates the notions of translation boundedness and Leptin density. This also motivates the notion of Delone measure. Part (a) of the lemma is analogous to \cite[Cor.~2]{G18}.
\begin{lemma}\label{lem:tbD}
Let $G$ be a unimodular locally compact group, and let $\nu$ be any measure on $G$. Then the following hold.
\begin{itemize}
\item[(a)] $\nu$ is upper translation bounded if and only if $\Lep_\nu^+$ is finite.
\item[(b)] $\nu$ is lower translation bounded if and only if $\Lep_\nu^-$ is positive.
\item[(c)] $\nu$ is a Delone measure if and only if $\Lep_\nu^-$ is positive and $\Lep_\nu^+$ is finite.
\end{itemize}
\end{lemma}
\begin{proof} (c) follows from (a) and (b). The  claims (a $\Rightarrow$) and (b $\Rightarrow$) follow from Lemma~\ref{lem:seLepDel}. In order to prove the remaining claims, fix a Haar measure $m$ on $G$ .

\noindent (a $\Leftarrow$) Assume $\Lep_\nu^+<\infty$. Then there exist a compact set $K$ and a finite constant $C$ such that $\nu(A)\le C m(KA)$ for all compact sets $A$ of positive measure. Choosing $A=Bx$ for some compact symmetric unit neighbouhood $B$, the claim follows from unimodularity of $G$.

\noindent (b $\Leftarrow$) Assume that $\Lep_\nu^->0$. Then there exist a compact set $K$ and a positive constant $C$ such that $\nu(KA)\ge C m(A)$ for all compact sets $A$ of positive measure. Choosing $A=Bx$ for some compact symmetric unit neighbouhood $B$, the claim follows from unimodularity of $G$.
\end{proof}

Due to the above standard estimates, Leptin densities of Delone measures characterise amena\-bility of unimodular locally compact groups.

\begin{lemma}\label{lem:elprop}
Let $G$ be a unimodular locally compact group. Let $\nu$ be any Delone measure on $G$. Then the following assertions are equivalent.
\begin{itemize}
\item[(i)] $G$ is amenable.
\item[(ii)] $\Lep^-_\nu\le \Lep^+_\nu$.
\item[(iii)] $\Lep^-_{\nu}$ is finite.
\item[(iv)] $\Lep^+_{\nu}$ is positive.
\end{itemize}
\end{lemma}
\begin{proof}
The implication (i) $\Rightarrow$ (ii) is Lemma~\ref{lem:Smp}. For the implication (ii) $\Rightarrow$ (iii) assume $\Lep^-_\nu=\infty$. Then (ii) implies $\Lep^+_\nu=\infty$. Hence by Lemma~\ref{lem:seLepDel} we have $I(G)=0$ and thus $L^-_\nu=0$, which is a contradiction. For the implication (iii) $\Rightarrow$ (iv), note that (iii) excludes the case $I(G) = \infty$ by   Lemma~\ref{lem:seLepDel}. In view of Lemma~\ref{ref:lemcharamen}, we have $I(G) = 1$.
But then $\Lep^+_\nu>0$ by Lemma~\ref{lem:seLepDel}.  For the implication (iv) $\Rightarrow$ (i), note that (iv) excludes the case $I(G) = \infty$ by Lemma~\ref{lem:seLepDel}, hence  $I(G)=1$ by Lemma~\ref{ref:lemcharamen}. Thus, $G$ is amenable by Theorem~\ref{theo:am}.
\end{proof}

\section{Leptin densities and Beurling densities}

\subsection{Leptin densities and Beurling densities}

Often uniform densities are used,  which are invariant under shifting the averaging sequence or the measure. These densities are sometimes called Beurling densities or Beurling-Landau densities, see \cite[p.~346]{B89} and \cite[p.~47]{L67} for $G=\mathbb R^d$. In ergodic theory they are known by the name Banach densities. In fact these densities can be defined without resorting to an averaging sequence. For discrete groups, this is studied in \cite{DHZ15}. The case $G=\mathbb R^d$ is considered in \cite{G13} by a linear functional approach.

\begin{defi}[Beurling densities]\label{prop:sfs}
Assume that $G$ is a unimodular locally compact group with left Haar measure $m$.
Let $\nu$ be a measure on $G$  and consider
\begin{displaymath}
B^-_\nu=\sup_{A\in\cK_p} \inf_{s\in G} \frac{\nu(As)}{m(A)}\ ,\qquad
B^+_\nu=\inf_{A\in \cK_p}\sup_{s\in G} \frac{\nu(As)}{m(A)} \ .
\end{displaymath}
We call $B_{\nu}^{-}$ and $B_{\nu}^{+}$ the lower Beurling density respectively the upper Beurling density of the measure $\nu$. If $B_{\nu}^{-}=B_{\nu}^{+}=B_\nu$, we call $B_\nu$ the Beurling density of $\nu$.
\end{defi}

In the following, we will suppress the subscript $\nu$ if the reference to the measure is clear.

\begin{remark}\label{rem:Bfin}
Note that $B^+_\nu<\infty$ if $\nu$ is upper translation bounded, as a consequence of Lemma~\ref{lem:tbb}.
The attribution to Beurling is motivated by Proposition~\ref{def:bdensities} below, which states that, for unimodular groups, the above densities coincide with a definition based on averaging.
\end{remark}

Let us discuss the relation between Beurling densities and Leptin densities.
Beyond the amenable situation, these notions differ. Indeed, for any left Haar measure $m$ on a unimodular non-amenable group $G$, we have $B^{+}_{m} = B^{-}_{m} = 1$, but $\Lep^-_m=\infty$ and $\Lep^+_m=0$.
It follows from the next theorem that, for a locally finite measure $\nu$ on a unimodular amenable group, the apriori different notions of Beurling density and Leptin density coincide. Precisely, the lower Beurling density of $\nu$ coincides with its lower Leptin density and the upper Beurling density of $\nu$ coincides with its upper Leptin density.

\begin{theorem}\label{BduD1}
Assume that $G$ is a unimodular locally compact group.  Consider any locally finite measure  $\nu$ on $G$. For the densities associated to $\nu$ we then have
\begin{displaymath}
B^-_\nu\le \Lep^-_\nu\le B^-_\nu \cdot I(G)\ , \qquad  I(G)^{-1} \cdot B^+_\nu \le \Lep^+_\nu   \le B^+_\nu \ .
\end{displaymath}
In particular if $G$ is additionally amenable, then we have $B^-_\nu=\Lep^-_\nu\le \Lep^+_\nu=  B^+_\nu$.
\end{theorem}

\subsection{Proof of Theorem~\ref{BduD1}}

We will treat the four inequalities separately. The inequalities $\Lep^-\le B^-\cdot I(G)$ and $B^+\le \Lep^+\cdot I(G)$ follow from elementary estimates.
\begin{proof}[Proof of $\Lep^-\le B^-\cdot I(G)$]
Consider any $A \in \mathcal K_p$, any nonempty $K \in \mathcal K$ and any $s\in G$. We have
\begin{displaymath}
\frac{\nu(KAs)}{m(A)}=\frac{m(KA)}{m(A)} \cdot \frac{\nu(KAs)}{m(KA)} \ .
\end{displaymath}
Taking the infimum over $A\in \mathcal K$ on the lhs and using unimodularity of $G$, we see that the lhs is independent of $s\in G$. Hence we may take the infimum over $s$ on the rhs and arrive at
\begin{displaymath}
\inf_{A\in \mathcal K_p}\frac{\nu(KA)}{m(A)}\le \frac{m(KA)}{m(A)} \cdot  \inf_{s\in G}\frac{\nu(KAs)}{m(KA)}
\le \frac{m(KA)}{m(A)} \cdot B^- \ .
\end{displaymath}
As the lhs is independent of $A\in \mathcal K_p$, we may take the infimum over $A$ on the rhs and the supremum over $K$ on both sides to conclude
\begin{displaymath}
\Lep^-=\sup_{K\in \mathcal K} \inf_{A\in \mathcal K_p}\frac{\nu(KA)}{m(A)}\le
\sup_{K\in \mathcal K}\inf_{A\in \mathcal K_p}\frac{m(KA)}{m(A)} \cdot B^- \le I(G)\cdot B^-\ .
\end{displaymath}
\end{proof}

\begin{proof}[Proof of $I(G)^{-1}\cdot B^+\le \Lep^+$]
Consider any $A \in \mathcal K_p$, nonempty $K \in \mathcal K$   and any $s\in G$ and note that by unimodularity of $G$ we have
\begin{displaymath}
\frac{m(A)}{m(KA)} \cdot \frac{\nu(As)}{m(A)} =\frac{\nu(As)}{m(KA)} \le \sup_{A'\in \mathcal K_p} \frac{\nu(A')}{m(KA')} \ ,
\end{displaymath}
where the rhs is independent of $s\in G$. Hence we may take the supremum over $s$ on the lhs and arrive at
\begin{displaymath}
\frac{m(A)}{m(KA)} \cdot  B^+\le\frac{m(A)}{m(KA)} \cdot \sup_{s\in G}\frac{\nu(As)}{m(A)} \le \sup_{A'\in \mathcal K_p} \frac{\nu(A')}{m(KA')} \ .
\end{displaymath}
As the rhs is independent of $A\in \mathcal K_p$,
we may take the supremum over $A$ on the lhs and the infimum over $K$ on both sides to conclude
$I(G)^{-1} \cdot  B^+\le \Lep^+$.
\end{proof}

The proofs of $B^-\le \Lep^-$ and of $\Lep^+\le B^+$ are more delicate.  We will follow ideas from \cite[Lem.~2.9]{DHZ15}. Note the following two elementary results.

\begin{lemma}\label{lem:simpobs}
Let $G$ be a unimodular locally compact group. Fix a left Haar measure $m$ on $G$ and let $\nu$ be a locally finite measure on $G$.
Then we have for any compact sets $A,B\subseteq G$ that
\begin{equation}\label{eq1}
\int_A \nu(aB)\, {\rm d}m(a)=\int_B \nu(Ab)\, {\rm d}m(b) \ .
\end{equation}
\end{lemma}

\begin{proof}
We evaluate the lhs and use Tonelli's theorem, which is applicable as $\nu$ is assumed to be locally finite, and $A$ and $B$ are compact sets. We obtain
\begin{displaymath}
\begin{split}
\int_A \nu(aB)\, {\rm d}m(a) &=
\int_A \int_{aB} {\rm d}\nu(x)\, {\rm d}m(a)
= \int_G \int_{A\cap xB^{-1}}\!\!\!\!\! {\rm d}m(a)\, {\rm d}\nu(x)=\int_G m(A\cap xB^{-1})\, {\rm d}\nu(x) \ ,
\end{split}
\end{displaymath}
where we used $x\in aB$ iff $a\in xB^{-1}$.
Evaluating the rhs and using Tonelli's theorem yield
\begin{displaymath}
\begin{split}
\int_B \nu(Ab)\, {\rm d}m(b) &=
\int_B \int_{Ab} {\rm d}\nu(x)\, {\rm d}m(b)
= \int_G \int_{B\cap A^{-1}x}\, {\rm d}m(b)\, {\rm d}\nu(x)=\int_G m(B\cap A^{-1}x)\, {\rm d}\nu(x) \ ,
\end{split}
\end{displaymath}
where we used $x\in Ab$ iff $b\in A^{-1}x$.
Now equality follows from unimodularity and left invariance of the Haar measure. Indeed we have
\begin{displaymath}
m(B\cap A^{-1}x)=m(B^{-1}\cap x^{-1}A)=m(xB^{-1}\cap A) \ .
\end{displaymath}
\end{proof}

\begin{lemma}\label{lem:ineq}
Let $G$ be a unimodular locally compact group. Fix a left Haar measure $m$ on $G$ and let $\nu$ be a locally finite measure on $G$.
Let $A,B\subseteq G$ be compact sets. Assume that there is a positive finite constant $C=C(A,B)$ such that for all $b\in B$ we have
\begin{displaymath}
\nu(Ab) \ge C\cdot m(A) \ .
\end{displaymath}
Then there exists $a\in A$ such that
\begin{displaymath}
\nu(aB) \ge C\cdot m(B) \ .
\end{displaymath}
\end{lemma}

\begin{remark}\label{lem:rev}
In the two inequalities of the above lemma, replace  ``$\ge$'' by ``$\le$''. Then the resulting statement remains true, as the following proof shows.
\end{remark}

\begin{proof}
Let $A,B$ be compact sets and assume that the assertion of the lemma is wrong. Then we have for all $a\in A$ that $\nu(aB)< C\cdot m(B)$. Integrate this inequality to obtain
\begin{equation}\label{eq2}
\int_{A}\nu(aB)\, {\rm d}m(a)< C \cdot m(A)\cdot m(B) \ .
\end{equation}
On the other hand, the inequality in the assumption of the lemma can be integrated to yield
\begin{equation}\label{eq3}
\int_B\nu(Ab) \, {\rm d}m(b) \ge  C \cdot m(A)\cdot m(B) \ .
\end{equation}
But the left hand sides of the above two inequalities are equal due to Lemma~\ref{lem:simpobs}. This is contradictory.
\end{proof}

\begin{proof}[Proof of $B^-\le \Lep^-$ ]
Recall from Definition~\ref{prop:sfs} that
\begin{displaymath}
B^-=\sup_{A\in\cK_p} \inf_{s\in G} \frac{\nu(As)}{m(A)}\ , \qquad \Lep^-=\sup_{K\in\cK}\inf_{A\in \cK_p}  \frac{\nu(KA)}{m(A)}
\ .
\end{displaymath}
Fix arbitrary $\varepsilon>0$ and take $A\in\cK_p$ such that for all $s\in G$ we have
\begin{displaymath}
B^--\varepsilon \le \frac{\nu(As)}{m(A)} \ .
\end{displaymath}
Let now $S\in \cK_p$ be arbitrary. Recall that $m(A)\cdot (B^--\varepsilon)\le \nu(As)$ for all $s\in S$. Hence by Lemma~\ref{lem:ineq} there exists $a\in A$ such that $m(S)\cdot (B^--\varepsilon)\le \nu(aS)\le \nu(AS)$. Hence we have
\begin{displaymath}
B^--\varepsilon \le \frac{\nu(AS)}{m(S)} \ .
\end{displaymath}
As $S\in\cK_p$ and $\varepsilon>0$ were arbitrary, the claim follows.
\end{proof}

\begin{proof}[Proof of $\Lep^+\le B^+$]
Recall from Definition~\ref{prop:sfs} that
\begin{displaymath}
\Lep^+=\inf_{K\in \cK}\sup_{A\in \cK_p}  \frac{\nu(A)}{m(KA)}\ , \qquad B^+=\inf_{A\in \cK_p} \sup_{s\in G} \frac{\nu(As)}{m(A)}\ .
\end{displaymath}
Assume without loss of generality that $B^+<\infty$. Fix arbitrary $\varepsilon>0$ and take $A\in \cK_p$ such that for all $s\in G$ we have
\begin{displaymath}
\frac{\nu(As)}{m(A)} \le B^++\varepsilon \ .
\end{displaymath}
Let now $S\in \cK_p$ be arbitrary. Since $\nu(At)\le (B^++\varepsilon)\cdot m(A)$ for all $t\in A^{-1}S$,  by Remark~\ref{lem:rev} there exists $a\in A$ such that $\nu(aA^{-1}S)\le (B^++\varepsilon)\cdot m(A^{-1}S)$. Since $S\subseteq a A^{-1}S$, we have
\begin{displaymath}
\frac{\nu(S)}{m(A^{-1}S)} \le B^++\varepsilon \ .
\end{displaymath}
As $S\in \cK_p$ and $\varepsilon>0$ were arbitrary, the claim follows.
\end{proof}

\section{F{\o}lner conditions and F{\o}lner nets}

As for F{\o}lner sequences or nets in locally compact groups there are various definitions in the literature depending on the definition of boundary. In this section we clarify some subtle connections, most of which can be found in the literature.

\subsection{Various boundaries}

The geometric intuition behind amenability is that, for suitable sets $A$, the measure of the boundary of $A$ is small when compared to the measure of $A$. Traditionally this is formalised using the so-called F{\o}lner boundaries
\[
\delta^K A := KA \, \triangle \, A = \big(KA \cap A^c \big) \cup \big( (KA)^c \cap A \big) \ .
\]
Here and also in the sequel of this text, $\triangle$ denotes the symmetric set difference, i.e., $A\, \triangle\, B=(A\cap B^c)\cup(A^c\cap B)$.  We write $A^c:= G \setminus A$ for the complement set of $A$.

\smallskip

As we discuss in the following sections, it is often useful to consider different notions of boundaries such as the {\em van Hove boundary}, defined as
\[
\partial^KA:=(KA\cap \overline{A^c})\cup (K^{-1} \overline{A^c}\cap A)
\]
or the {\em strong F{\o}lner boundary}, defined as
\[
\partial_K A:=K^{-1}A\cap K^{-1}A^c= \{g\in G: Kg \cap A\ne\varnothing \wedge Kg\cap A^c\ne\varnothing \}
\]
for sets $K, A \subseteq G$.
\begin{remark} \label{rem:FolvH}
	The attribution to van Hove is due to Schlottmann \cite{Martin2}. Van Hove boundaries are often used in model set analysis, see \cite{BaakeLenz2004} and \cite[Sec.~2]{mr13} for further background.
	Strong F{\o}lner boundaries have been introduced and used by Ornstein and Weiss in \cite{OW87}.
	The name strong F{\o}lner boundary is coined in \cite{PS16}. Also note the slightly different definition in \cite{H21}, which is more in line with the notion of van Hove boundary, compare \cite[Rem.~2.2(ii)]{H21}.
Such boundaries naturally arise in tiling problems leading to sub- or almost additive convergence results beyond Euclidean space, cf.\@ e.g.\@ \cite{LW00, Kr07, CCK14, PS16, H21}.
\end{remark}

\smallskip

Note that the van Hove boundary and the strong F{\o}lner boundary are both monotonic in $K$, in contrast to the classical F{\o}lner boundary. Strong F{\o}lner boundaries additionally satisfy the simple relation $L\partial_KA\subseteq \partial_{KL^{-1}}A$, which is very convenient in calculations.  Let us collect some relations between the three types of boundary for later use.

\begin{lemma}[Boundary comparison]\label{lem:vHsF}
	Let $G$ be a locally compact group.
	Let $K\subseteq G$ be any symmetric unit neighborhood and $A\subseteq G$ be arbitrary. Then
	\begin{displaymath}
	\partial_K A\subseteq \partial^KA\subseteq \partial_{K^2}A \ , \qquad
	\delta^KA\subseteq \partial_K A\subseteq K \delta^KA  \ ,\qquad
	\partial_K(KA)\subseteq \delta^{K^2}A  \ .
	\end{displaymath}
\end{lemma}

\begin{remark}
	The statements comparing the boundaries $\partial_K$ and $\delta^K$ can be found in \cite[Lem.~2.2]{Kr07}. The inclusion $\partial_K(KA) \subseteq \delta^{K^2}A$ has been observed in \cite{OW87} and leads to the uniform F{\o}lner condition (UFC) in \cite[p.~19]{OW87}. We provide the proofs for the sake of self-containment.
	\end{remark}

\begin{proof}
Let $K=K^{-1}$ be any symmetric unit neighborhood.
We have to show five inclusions which we enumerate by (i)--(v) in its order of appearance in the statement.

\noindent (i) To show $\partial_KA\subseteq \partial^KA$, let $x\in KA\cap KA^c$ and assume $x\notin KA\cap \overline{A^c}$. Then $x\in A$, which implies $x\in KA\cap KA^c\cap A=A\cap KA^c\subseteq A\cap K\overline{A^c}\subseteq \partial^KA$.

\noindent (ii) Noting $\overline{A^c}\subseteq KA^c$ we have $KA\cap \overline{A^c}\subseteq KA\cap KA^c\subseteq K^2A\cap K^2A^c$. Likewise we have $K\overline{A^c}\cap A\subseteq KKA^c\cap KA\subseteq K^2A^c\cap K^2 A$. We thus have $\partial^KA\subseteq \partial_{K^2}A$.

\noindent (iii) As $(KA)^c\subseteq KA^c$, we have
$\delta^K A=(KA\cap A^c)\cup ((KA)^c\cap A)\subseteq (KA\cap A^c)\cup (KA^c\cap A)\subseteq \partial_KA$.

\noindent (iv) We have $\partial_KA=KA\cap KA^c\cap (A\cup A^c)=(A\cap KA^c) \cup (KA\cap A^c)$. As we have $A\cap KA^c\subseteq K(KA\cap A^c)$,  this implies $\partial_KA\subseteq K\delta^KA$.

\noindent (v) We have $\partial_K(KA)=K (KA) \cap K(KA)^c =K^2A \cap K(KA)^c \subseteq K^2A \cap A^c \subseteq \delta^{K^2}A$.
\end{proof}

Let us denote by $\mathcal K_{s0}\subseteq \mathcal K$ the collection of compact sets which are symmetric and contain the identity. The following proposition characterizes amenability. It is well-known, we give a proof for the convenience of the reader.

\begin{prop}\label{prop:refine}
	Let $G$ be a locally compact group with left Haar measure $m$. Then the following are equivalent.
\begin{itemize}
\item[(i)]  For all $K\in\mathcal K$ we have $\inf_{A\in \mathcal K_p} m(KA)/m(A)=1$.
\item[(ii)]  For all $K\in\mathcal K$ we have  $\inf_{A\in \mathcal K_p} m(\delta^K A)/m(A)=0$.
\item[(iii)]  For all $K\in\mathcal K$ we have  $\inf_{A\in \mathcal K_p} m(\partial_K A)/m(A)=0$.
\item[(iv)]  For all $K\in\mathcal K$ we have  $\inf_{A\in \mathcal K_p} m(\partial^K A)/m(A)=0$.
\end{itemize}
In fact the claimed equivalences continue to hold if in any of the four statements (i) to (iv) the condition ``$K\in \mathcal K$'' is replaced by ``$K\in \mathcal K_{s0}$''.
\end{prop}

\begin{proof}
	We will first prove the claimed equivalences for $K\in \mathcal K_{s0}$.
	The equivalence of (i) and (ii)  is a direct consequence of $A \, \dot{\cup} \, \delta^K A= KA$, which holds as $e\in K$.
	For (ii) $\Rightarrow$ (iii), let $\varepsilon>0$ be arbitrary and choose $A'\in \mathcal K_p$ such that $m(\delta^{K^2} A')/m(A')<\varepsilon$. Letting $A=KA'$, we then estimate
	$\partial_K A = \partial_K(KA') \subseteq \delta^{K^2} A'$.
	As we have $m(A)=m(KA')\ge m(A')$, we obtain $m(\partial_KA)/m(A)\le m(\delta^{K^2}A')/m(A')<\varepsilon$. As $\varepsilon>0$ was arbitrary, then claim follows.

The implications (iii) $\Rightarrow$ (iv) and (iv) $\Rightarrow$ (ii) follow from the comparison lemma~ \ref{lem:vHsF}. Indeed that lemma states $\partial^KA\subseteq \partial_{K^2}A$ and $\delta^KA\subseteq \partial^{K}A$.
	
Consider now the general case $K\in \mathcal K$ in the above statements. Then statements $(i)$, $(iii)$, $(iv)$ hold if these statements hold for $K\in \mathcal K_{0s}$, as the corresponding measure ratios are monotonic in $K$. With respect to $(ii)$, note $\delta^KA\subseteq \partial_{K^{-1}\cup \{e\}}A$. Hence $(ii)$ follows from $(iii)$, which is already established.	
\end{proof}

\subsection{F{\o}lner nets}

Consider a Delone measure $\nu$ on an amenable locally compact group $G$. In that case, one might ask whether there exists a sequence $(A_n)_{n\in\mathbb N}$ of compact sets in $G$ such that the Leptin density $L_\nu^-$ is realised on that sequence, instead of taking the infimum over all $A\in \mathcal K_p$. Natural candidates are F{\o}lner sequences, which exist in any $\sigma$-compact amenable group. For general locally compact groups, amenability is characterized by the existence of F{\o}lner nets. We refer to \cite[Ch.~4]{P88} for the construction of F{\o}lner nets.
For the convenience of the reader, background about nets is collected in Appendix \ref{sec:nets}.

\begin{defi}[F{\o}lner net]\label{def:Folner}
Let $G$ be an amenable locally compact group with left Haar measure $m$. Let $(\mathbb I, \prec)$ be a directed poset and let $(A_i)_{i\in\mathbb I}\in \mathcal (\mathcal K_p)^{\mathbb I}$. Then $(A_i)_{i\in\mathbb I}$ is called a {\em F{\o}lner net} if
\begin{displaymath}
\sup_{K\in \mathcal K}\lim_{i\in \mathbb I} \frac{m(\delta^K A_i)}{m(A_i)} = 0   \ .
\end{displaymath}
\end{defi}

\begin{remark}
Note that any subnet of a F{\o}lner net is a F{\o}lner net.  
If $G$ is unimodular and $(A_i)_{i \in \mathbb{I}}$ is a F{\o}lner net, then $(A_ix_i)_{i\in \mathbb I}$ is a F{\o}lner net for every net $(x_i)_{i\in \mathbb I}\in G^{\mathbb{I}}$. This is a direct consequence of right-invariance of the Haar measure in that case.
For $\sigma$-compact amenable $G$, there always exists a F{\o}lner sequence. See the proof of \cite[Thm.~4.16]{P88}.
\end{remark}

Clearly, one can define alternative F{\o}lner nets or sequences using $\partial^K$ or $\partial_K$ as a notion of boundary. This has lead to the terminology of {\em van Hove sequences} or {\em strong F{\o}lner sequences} in the literature. Let us introduce these notions in the general context of nets. Namely, $(A_i)_{i \in \mathbb I} \in (\mathcal{K}_p)^{\mathbb I}$ is called {\em strong F{\o}lner net}, respectively a {\em van Hove net} if
\[
\sup_{K \in \mathcal{K}} \lim_{i \in \mathbb I} \frac{m\big( \partial_K A_i \big)}{m(A_i)} = 0, \quad \mbox{ respectively } \quad \sup_{K \in \mathcal{K}} \lim_{i \in \mathbb I} \frac{m\big( \partial^K A_i \big)}{m(A_i)} = 0.
\]

\medskip

It is obvious from the boundary comparison lemma~\ref{lem:vHsF} that the notions of strong F{\o}lner net and van Hove net coincide, a fact that has already been remarked in \cite[Rem.~2.2(ii)]{H21}. In the remainder, we will mostly work with strong F{\o}lner boundaries.
It is easy to see that every strong F{\o}lner net is a  F{\o}lner net, since  $\delta^K A \subseteq \partial_{K^{-1} \cup \{e\}} A$ for all $K \in \mathcal{K}$ and each $A \in \mathcal{K}_p$.  An example of a  F{\o}lner sequence $(A_n)_{n\in \mathbb N}$ in $G=\mathbb R$  that is not a strong F{\o}lner sequence can be provided by compact nowhere dense sets $A_n\subset[0,n]$ of Lebesgue measure $n-1/n$, see \cite{K14}. An example in $G=\mathbb R^d$ has been given in \cite[Appendix, Ex.~3.4]{T92}. The next proposition characterizes strong F{\o}lner nets. Whereas parts of it are known, we give a proof for the convenience of the reader.

\medskip

\begin{prop}[Characterization of strong F{\o}lner nets]\label{lem:refine2} Let $G$ be an amenable locally compact group with left Haar measure $m$.
Let $(A_i)_{i \in \mathbb I} \in \mathcal{K}_p^{\mathbb I}$.
Then the following are equivalent.
\begin{itemize}
\item[(i)] $(A_i)_{i\in \mathbb I}$ is a strong F{\o}lner net.
\item[(ii)] $(A_i)_{i \in \mathbb I}$ is a F{\o}lner net and there is an open unit neighborhood $O$ such that \\
\quad $\lim_{i \in \mathbb I} \frac{m\big( \bigcap_{o \in O} oA_i \big)}{m(A_i)} = 1.$
\item[(iii)] $(A_i)_{i \in \mathbb I}$ is a F{\o}lner net and  for every compact symmetric unit neighborhood $K$ \\
\quad $\lim_{i \in \mathbb I} \frac{m\big( \bigcap_{k \in K} kA_i \big)}{m(A_i)} = 1.$

\item[(iv)]  For all $K,L\in\mathcal K$ we have  $\lim_{i\in \mathbb I} \frac{m(L\delta^KA_i)}{m(A_i)} =0$.

\item[(v)]  For all $K,L\in\mathcal K$ we have  $\lim_{i\in \mathbb I} \frac{m(L\partial^KA_i)}{m(A_i)} =0$.

\item[(vi)]  For all $K,L\in\mathcal K$ we have  $\lim_{i\in \mathbb I} \frac{m(L\partial_KA_i)}{m(A_i)}   =0$.

\end{itemize}
In fact the same equivalences (iv), (v) and~(vi) hold if ``$K, L\in \mathcal K$'' is replaced by ``$K, L\in \mathcal K_{s0}$'' in any of the above  statements.
In particular, $(A_i)_{i \in \mathbb{I}}$ is a van Hove net (sequence) if and only if it is a strong F{\o}lner net (sequence).
If $G$ is additionally assumed to be discrete, then every F{\o}lner net (sequence) is also a strong F{\o}lner net (sequence).
\end{prop}

\begin{remark}
The equivalence (i)$\Leftrightarrow$(ii) already appears in \cite[Appendix, (3.K)]{T92}. It shows that F{\o}lner nets are strong F{\o}lner nets if the they display some asymptotic invariance with respect to small topological pertubations. In this situation, one gets with (iii)  asymptotic invariance for the ``inner part'' $A_i \setminus \mathrm{Int}_K(A_i)$, where $\mathrm{Int}_K(A) = \{g \in A:\, Kg \subseteq A\}$ for a compact symmetric unit neighborhood $K$ and a set $A$.
	The observation that for discrete groups, the notions of a F{\o}lner net (or sequence) and a strong F{\o}lner net (or sequence) are equivalent has also been observed in \cite[Prop.~2.3]{Kr07} and \cite[Prop.~2.4]{CCK14}. Note also that the authors deal in \cite{CCK14} with semigroups and with boundary defined as the inner part of the strong F{\o}lner boundary, i.e. \@ $\widetilde{\partial}_K A := A \setminus \mathrm{Int}_K(A)$.
\end{remark}

\begin{proof}
	We have indicated above that every strong F{\o}lner net is a F{\o}lner net. In order to complete the proof of the implication (i)$\Rightarrow$(ii), let $O$ be an open relatively compact, symmetric unit neighborhood.
	Since $A_i \setminus \bigcap_{o \in O} oA_i \subseteq \partial_O(A_i)$, and since $(A_i)$ is a strong F{\o}lner net by (i), we obtain
	\[
	\liminf_{i \in \mathbb I} \frac{m\big( \bigcap_{o \in O} oA_i \big)}{m(A_i)} \geq \Bigg(  1 - \limsup_{i \in \mathbb I}  \frac{m\big( \partial_O A_i \big)}{m(A_i)}\Bigg) = 1.
	\]
	We turn to the proof of (ii)$\Rightarrow$(iii). Let $K$ be any compact symmetric unit neighborhood.
	We find a finite cover of $K$ by left-translates $s_j O$. Then
	\begin{align*}
	\limsup_{i \in \mathbb I} \frac{m\big( A_i \setminus \bigcap_{k \in k} kA_i \big)}{m(A_i)}
	&\leq \limsup_{i \in \mathbb I} \frac{m\big( A_i \setminus \bigcap_{j} \bigcap_{o \in O} s_j oA_i \big)}{m(A_i)} \\
	&\leq \limsup_{i \in \mathbb I} \sum_j \frac{m\big( s_j^{-1} A_i \setminus \bigcap_{o \in O} oA_i \big)}{m(A_i)}  \\
	&\leq \sum_j \Bigg( \limsup_{i \in \mathbb I} \frac{m\big( s_j^{-1}A_i \,\triangle \, A_i \big)}{m(A_i)} + \limsup_{i \in \mathbb I} \frac{m\big( A_i \setminus \bigcap_{o \in O} o A_i \big)}{m(A_i)}  \Bigg) \\
	&= 0.
	\end{align*}
	Note that we used for the last equality that $(A_i)$ is a F{\o}lner net and the assumption on $O$. This shows the assertion~(iii). We turn to the proof of the implication (iii)$\Rightarrow$(i). By Lemma~\ref{lem:vHsF} and the fact that $(A_i)$ is a F{\o}lner net,
	\begin{align*}
	\limsup_{i \in \mathbb I} \frac{m\big( \partial_K A_i \big)}{m(A_i)} &\leq \limsup_{i \in \mathbb I} \frac{m\big( K(KA_i \setminus A_i) )}{m(A_i)} \leq \limsup_{i \in \mathbb I} \frac{m\big( K^2 A_i \setminus \bigcap_{k \in K} kA_i \big)}{m(A_i)} \\
	&\leq \limsup_{i \in \mathbb{I}} \frac{m(K^2 A_i \setminus A_i)}{m(A_i)} + \limsup_{i \in \mathbb{I}} \frac{m( A_i \setminus \bigcap_{k \in K} kA_i)}{m(A_i)} \\
	&=0.
	\end{align*}
	Due to the monotonicity of the strong F{\o}lner boundary in $K$, the same is true for general $K \in \mathcal{K}$. Hence $(A_i)$ is a strong F{\o}lner net. The equivalence of (i) and (vi) is due to the relation $L\partial_K A \subseteq \partial_{KL^{-1}}A$ for general compact sets $K,L,A$. If restricting oneself to compact symmetric unit neighborhoods $K$ and $L$, the equivalences (iv)$\Leftrightarrow$(v)$\Leftrightarrow$(vi) follow from the boundary comparison lemma, Lemma~\ref{lem:vHsF}. Due to monotonicity, these assertions are equivalent when formulated for general $K,L \in \mathcal{K}$. Assuming in addition that $G$ is discrete, then the implication (ii) $\Rightarrow$ (i) shows that every F{\o}lner net is a strong F{\o}lner net.
\end{proof}

It is well known  that existence of a  F{\o}lner sequence implies existence of a strong F{\o}lner sequence, cf.\@ \cite[Lem.~2.6]{PS16} or \cite[Lem.~2.2]{K14}. Here we describe how a strong F{\o}lner net can be constructed from any  F{\o}lner net by a simple ``thickening procedure''.

\begin{prop}[Construction of strong F{\o}lner nets]
	Let $G$ be an amenable locally compact group with left Haar measure $m$. If $(A_i)$ is a F{\o}lner net and $L$ is a compact symmetric unit neighborhood, then the net $(A_i^{\prime})$ with $A_i^{\prime} := LA_i$ is a strong F{\o}lner net.
\end{prop}

\begin{proof}
	We first note that $(A_i^{\prime})$ is a F{\o}lner net which follows from the  F{\o}lner property of $(A_i)$ and the inclusion
	\[
	A_i^{\prime} \, \triangle \, KA_i^{\prime} \subseteq \big( LA_i \,\triangle \, A_i\big) \cup \big( A_i \setminus KA_i \big) \cup \big( KLA_i \setminus A_i \big)
	\]
	for arbitrary $K \in \mathcal{K}$.
	Further, take an open symmetric unit neighborhood $O \subseteq L$. We claim that $A_i \subseteq \bigcap_{o \in O} oA_i^{\prime}$. Indeed, for any $a \in A_i$ and each $o \in O$, we have $a = oo^{-1}a \subseteq oA_i^{\prime}$. Using that $(A_i)$ is a F{\o}lner net, we obtain
	\[
	\liminf_{i \in \mathbb I} \frac{m\big( \bigcap_{o \in O} oA_i^{\prime} \big)}{m(A_i^{\prime})} \geq \liminf_{i \in \mathbb{I}} \frac{m(A_i)}{m(LA_i)} = 1.
	\]
	Now the claim follows from the previous proposition, equivalence (i)$\Leftrightarrow$(ii).
\end{proof}

The following lemma shows that (left)-asymptotic invariance observed via strong F{\o}lner nets carries over to upper translation bounded measures.
For $\sigma$-compact LCA groups, weaker versions of it are well known, see e.g.~\cite[Prop.~9.1]{LR07} and \cite[Lem.~2.2, Prop.~5.1]{mr13}.

\smallskip

\begin{lemma}\label{lem:assm}
Let $G$ be a unimodular amenable locally compact group with Haar measure $m$.  Let $(\mathbb I, \prec)$ be a directed poset and let $(A_i)_{i\in\mathbb I}$ be a strong F{\o}lner net.  Let $\nu$ be a measure on $G$ that is upper translation bounded. We then have
\begin{displaymath}
\sup_{K,L\in \mathcal K}\lim_{i\in \mathbb I} \sup_{s\in G} \frac{\nu(L \delta^K A_is)}{m(A_i)} =
\sup_{K,L\in \mathcal K}\lim_{i\in \mathbb I} \sup_{s\in G}\frac{\nu(L \partial_K A_is)}{m(A_i)} =
\sup_{K,L\in \mathcal K}\lim_{i\in \mathbb I} \sup_{s\in G}\frac{\nu(L \partial^K A_is)}{m(A_i)} = 0   \ .
\end{displaymath}
\end{lemma}

\begin{proof}
We use Lemma~\ref{lem:tbb} to transfer the properties in Lemma~ \ref{lem:refine2} to the measure $\nu$. For instance, for the strong F{\o}lner boundary this yields
\begin{displaymath}
\lim_{i\in \mathbb I} \sup_{s\in G}\frac{\nu(L \partial_K A_is)}{m(A_i)}
\le \lim_{i\in \mathbb I}  \frac{C_u}{m(B_u)}\frac{m(B_u L \partial_K A_i)}{m(A_i)}  =0 \ ,
\end{displaymath}
where we used right invariance of the Haar measure in the estimate.
\end{proof}

\subsection{Densities on F{\o}lner nets in unimodular groups}

In this section, we restrict ourselves to amenable locally compact groups that are unimodular.  We define densities of a measure via averaging with respect to a F{\o}lner net.

\begin{defi}[$\cA$-densities] \label{def:adensities}
Let $G$ be a unimodular amenable locally compact group with left Haar measure $m$.
Let $\nu$ be a measure on $G$, and let $\cA=(A_i)_{i\in\mathbb I}$ be a strong F{\o}lner net in $G$.  We define the upper and lower densities of $\nu$ with respect to $\cA$ by
\begin{displaymath}
\begin{split}
D^-_\cA= \liminf_{i\in \mathbb I} \frac{1}{m(A_i)}  \nu( A_i) \ , &\qquad
D^+_\cA= \limsup_{i\in\mathbb I} \frac{1}{m(A_i)} \nu( A_i) \\
B^-_\cA= \liminf_{i\in \mathbb I} \frac{1}{m(A_i)}  \inf_{s\in G} \nu( A_is) \ ,
&\qquad B^+_\cA= \limsup_{i\in\mathbb I} \frac{1}{m(A_i)} \sup_{s\in G}\nu( A_is)  \ .
\end{split}
\end{displaymath}
If $D^-_\cA=D^+_\cA=D_\cA$, we say that $\nu$ has $\cA$-density $D_\cA$.
\end{defi}

\begin{remark}
We have $0\le B^-_\cA\le D^-_\cA\le D^+_\cA\le B^+_\cA\le\infty$ by definition. If $\nu$ is upper translation bounded, we have $B^+_\cA<\infty$.  If $\nu$ is lower translation bounded, we have $B^-_\cA>0$.
\end{remark}

\begin{prop}\label{def:bdensities}
Assume that $G$ is a unimodular amenable locally compact group with left Haar measure $m$.
Let $\nu$ be an upper translation bounded measure on $G$. Take any strong F{\o}lner net  $\cA=(A_i)_{i\in\mathbb I}$ in $G$. We then have
\begin{displaymath}
B^-_\cA= \lim_{i\in \mathbb I} \frac{1}{m(A_i)}  \inf_{s\in G} \nu( A_is) \ ,
\qquad
B^+_\cA= \lim_{i\in\mathbb I} \frac{1}{m(A_i)} \sup_{s\in G}\nu( A_is) \ .
\end{displaymath}
Moreover we have
\begin{displaymath}
0 \le L^-=B^-=B^-_\cA \le D^-_\cA\le D^+_\cA\le  B^+_\cA=B^+=L^+  <\infty \ ,
\end{displaymath}
where $L^-,L^+$ are the Leptin densities from Definition~\ref{def:lep}, and where $B^-, B^+$ are the Beurling densities from Definition \ref{prop:sfs}.
In particular, the above $\mathcal A$-densities do not depend on the choice of the strong F{\o}lner sequence.
\end{prop}

\begin{remark}
In the following proof, note that upper translation boundedness and the strong version of the F{\o}lner property only enter in showing that $B^-_\cA=B_-$. Assuming only local finiteness of $\nu$ we still have $B^-_\cA \le B^-=L^-\le L^+=  B^+ \le B^+_\cA$.
\end{remark}

\begin{proof}
Consider
\begin{displaymath}
C^-_\cA= \limsup_{i\in \mathbb I} \frac{1}{m(A_i)}  \inf_{s\in G} \nu( A_is) \ ,
\qquad
C^+_\cA= \liminf_{i\in\mathbb I} \frac{1}{m(A_i)} \sup_{s\in G}\nu( A_is) \ .
\end{displaymath}
(i) Note that for each $i \in \mathbb I$ we have
$$
\frac{1}{m(A_i)}  \inf_{s\in G} \nu( A_is)\leq \sup_{A \in \cK_p} \frac{1}{m(A)}  \inf_{s\in G} \nu( As) = B^- \ ,
$$
which gives $C^-_\cA\le B^-$. If $\Lep^-\le C^-_\cA$, then $C^-_\cA= B^-=L^-$ by Theorem~\ref{BduD1}. To show  $\Lep^-\le C^-_\cA$, we take arbitrary $d<\Lep^-$. Take $K \in \mathcal K$ such that for all $A\in\mathcal K_p$ we have $d\le \nu(KA)/m(A)$.
Choosing $A=A_is$ for $i\in\mathbb I$ we obtain the estimate
\begin{displaymath}
d\le \frac{\nu(KA_is)}{m(A_i)} \le
\frac{\nu(A_is)}{m(A_i)}+
\frac{\nu((\delta^{K}A_i)s)}{m(A_i)} \, ,
\end{displaymath}
where we used the estimate $KAs\subseteq As \cup (\delta^{K}A)s$, and the unimodularity of $G$. Now upper translation boundedness of $\nu$, right invariance of the Haar measure and the strong F{\o}lner property of $\cA$ yield $d\le C^-_\cA$, cf.\@ Lemma~\ref{lem:assm}. As $d<\Lep^-$ was arbitrary, we infer $\Lep^-\le C^-_\cA$. Therefore, we have $C^-_\cA= B^-=L^-$.
\smallskip

\noindent (ii)
Since for all $i\in \mathbb I$ we have
$$
 \frac{1}{m(A_i)} \sup_{s\in G}\nu( A_is) \geq \inf_{A \in \cK_p}  \frac{1}{m(A)} \sup_{s\in G}\nu( As) = B^+ \,,
$$
we get $B^+\le C^+_\cA$. If $C^+_\cA\le \Lep^+$, then $C^+_\cA= B^+=L^+$ by Theorem~\ref{BduD1}, which uses local finiteness of $\nu$. Let us thus show $\Lep^+\ge C^+_\cA$. Assume without loss of generality $\Lep^+<\infty$. Assume $d>\Lep^+$ for some finite $d$. Take $K \in \mathcal K$ such that for all $A\in \mathcal K_p$ we have $d\ge \nu(A)/m(KA)$. Choosing $A=A_is$ for  $i\in\mathbb I$ we obtain with the unimodularity of $G$ the estimate
\begin{displaymath}
d\ge \frac{\nu( A_is)}{m(KA_i)} \ .
\end{displaymath}
By the F{\o}lner property of $\cA$  this implies $d\ge C^+_\cA$. As $d>\Lep^+$ was arbitrary, we infer $\Lep^+\ge C^+_\cA$. Therefore, we have $C^+_\cA= B^+=L^+$.

\smallskip

\noindent (iii) By (i) and (ii), the values $C^-_\cA$ and $C^+_\cA$ are independent of the choice of the strong F{\o}lner net. This means that $C^-_\cA$ and $C^+_\cA$ are limits. In particular this implies $C^-_\cA=B^-_\cA \le D^-_\cA \le D^+_\cA \le B^+_\cA =C^+_\cA$. Also note $B^+<\infty$ by upper translation boundedness of $\nu$, see Remark~\ref{rem:Bfin}.
\end{proof}

\section{Leptin densities for lattice point counting}\label{sec:sud}

\noindent
We say that $\Gamma\subseteq G$ is a uniform lattice in a locally compact group $G$ if $\Gamma$ is a countable discrete cocompact subgroup of $G$.
Recall that existence of a uniform lattice in $G$ implies that $G$ is unimodular, see e.g.~\cite[Thm.~9.1.6]{DE}. A lattice is both left-uniformly discrete and right-uniformly discrete.
Analogously, a uniform lattice $\Gamma$ in $G$ is both left-relatively dense and right-relatively dense. Hence it admits a measurable, relatively compact left-fundamental domain and also a measurable, relatively compact right-fundamental domain. For a proof, see Theorem 1 and Remark 6 in \cite{fg68}.
 Since we are only interested in the covolume of a fundamental domain we can make this choice freely and fix a  relatively compact measurable left-fundamental domain $F$, i.e., \@ we have $\Gamma F = G$ and $F^{-1}\gamma_1\cap F^{-1}\gamma_2\neq \varnothing$ for some $\gamma_1,\gamma_2\in \Gamma$ implies $\gamma_1=\gamma_2$.
 Moreover, for any $\gamma \in \Gamma$, we observe that 
 \begin{align} \label{eqn:FD}
 \operatorname{card}(\Gamma \cap F^{-1}\gamma) = \operatorname{card}(\Gamma \gamma^{-1} \cap F^{-1}) = \operatorname{card}(\Gamma \cap F) = 1.
 \end{align}
 This fact will be used below in the proofs of the density formulas.
 	
\smallskip

\noindent
When identified with the Delone measure $\delta_\Gamma = \sum_{\gamma \in \Gamma} \delta_{\gamma}$,
uniform lattices $\Gamma$  have a positive finite Leptin density, which equals the reciprocal of its covolume $\mathrm{covol}(\Gamma)$. Recall that the covolume is defined as the Haar measure of any measurable relatively compact (left or right) fundamental domain.

\begin{prop}[Density of a uniform lattice]\label{prop:sudl}
Let $\Gamma$ be a uniform lattice in a unimodular amenable locally compact group $G$. Then $\Gamma$ has positive finite Leptin density
\begin{displaymath}
\Lep_\Gamma := \Lep_{\delta_\Gamma} =\frac{1}{m(F)} = \frac{1}{\mathrm{covol}(\Gamma)}\ ,
\end{displaymath}
where $F$ is any measurable relatively compact (left or right) fundamental domain for $\Gamma$.
\end{prop}

\begin{proof}
Fix a measurable relatively compact left-fundamental domain $F$ of $\Gamma$.
Define the compact set $K=\overline{F^{-1}F}$. Fix an arbitrary compact set $A\subseteq G$ and consider
\begin{displaymath}
\Gamma_A=\{\gamma\in \Gamma: A\cap F^{-1}\gamma \ne\varnothing\}\ , \qquad A_F=\bigcup_{\gamma\in \Gamma_A} F^{-1}\gamma\ .
\end{displaymath}
Note that $\Gamma_A$ is a finite set and $A_F$ is a disjoint union. We have
\begin{equation}\label{eq:afd}
A\subseteq A_F\subseteq KA\ .
\end{equation}
For the second inclusion, assume that $x\in A_F$. Then $x\in F^{-1}\gamma$ for some $\gamma\in \Gamma_A$ and there exists some $a\in A$ such that $a\in F^{-1}\gamma$. Hence $x\in F^{-1}\gamma \subseteq F^{-1} Fa\subseteq F^{-1}FA\subseteq KA$.
With equality~\eqref{eqn:FD} we thus have
\begin{displaymath}
\begin{split}
m(A)&\le m(A_F)= m(F)\cdot \mcard(\Gamma_A)=m(F)\cdot \sum_{\gamma\in \Gamma_A} \mcard(\Gamma\cap F^{-1}\gamma)
=m(F)\cdot \mcard(\Gamma\cap A_F)\\
&\le m(F)\cdot \mcard(\Gamma\cap KA) = m(F) \cdot \delta_\Gamma\big( KA \big).
\end{split}
\end{displaymath}
By Definition~\ref{def:lep} this implies $\Lep^-\ge 1/m(F)$. We can similarly estimate
\begin{displaymath}
\begin{split}
m(KA)&\ge m(A_F)= m(F)\cdot \mcard(\Gamma_A)=m(F)\cdot \mcard(\Gamma\cap A_F)\ge m(F)\cdot \delta_\Gamma(A) \ .
\end{split}
\end{displaymath}
This leads to $\Lep^+\le 1/m(F)$ by Definition~\ref{def:lep}. Finally, since $G$ is amenable, we can use Lemma~\ref{lem:Smp} in order to conclude  $\Lep^-=\Lep^+$. This and the fact that the expression $1/m(F)$ does not depend on $F$ show that $\Gamma$ has Leptin density $1/\mathrm{covol}(\Gamma) = 1/m(F)$.
\end{proof}

The next result, whose proof uses a standard argument \cite[Lem.~2]{Sie43}, will be needed in the following section.
\begin{lemma}\label{prop:exfd}
Let $G$ be a locally compact group and let $\Gamma$ be a uniform lattice in $G$. Consider any measurable relatively compact set $U\subseteq G$ such that $\Gamma U =G$. Then $U$ contains a measurable relatively compact left-fundamental domain for $\Gamma$ in $G$.
\end{lemma}

\begin{proof}
Let $F$ be a measurable relatively compact left-fundamental domain for $\Gamma$ in $G$. By relative compactness of both $F$ and $U$, the set $\Gamma_U=\{\gamma\in \Gamma:U \cap \gamma F \ne\varnothing\}$ is finite, and we can thus write $\Gamma_U=\{\gamma_1,\ldots,\gamma_n\}$.
Define
 $F_k= F\cap \gamma_k^{-1} U\subseteq F$ and note $F=\bigcup_{k=1}^n F_k$ as $\Gamma U=G$.
Define $F_1'= \gamma_1 F_1$ and
$F_k'=(\gamma_kF_k)\setminus \bigcup_{j=1}^{k-1}( \Gamma F_j)\subseteq U$ for $k \geq 2$ and note that $F_U:=\dot{\bigcup}_{k=1}^n F_k'$ is a disjoint union. Furthermore, $F_U$ is a measurable relatively compact set which satisfies $F_U\subseteq U$.
We show that $F_U$ is a left-fundamental domain for $\Gamma$. Consider any $z\in G$ and choose $k\in \{1,\ldots, n\}$ such that $z\in \Gamma F_k$ but $z\notin \Gamma F_j$ for all $j<k$. Then $z\in \Gamma F_k' \subseteq \Gamma F_U$ and thus $\Gamma F_U=G$. Assume $F_U^{-1}\gamma_1^{\prime} \cap F_U^{-1}\gamma_2^{\prime} \neq \varnothing$ for some $\gamma_1^{\prime}, \gamma_2^{\prime}\in \Gamma$. Then $F_r'^{-1}\gamma_1^{\prime} \cap F_s'^{-1}\gamma_2^{\prime} \neq\varnothing$ for some $r$ and some $s$. We may assume $r\le s$ without loss of generality. By construction of $F_s'$ we then have $r=s$. Since $F_r'\subseteq \gamma_r F$, we have $F^{-1} \gamma_r^{-1}\gamma_1^{\prime}\cap F^{-1}\gamma_r^{-1}\gamma_2^{\prime}\neq\varnothing$. Since $F$ is a left-fundamental domain for $\Gamma$, this implies $\gamma^{\prime}_1=\gamma^{\prime}_2$.
\end{proof}

\section{Uniform model sets in amenable groups}

We analyse the  upper and lower Leptin densities of weak model sets in amenable groups. As an application we give a simple geometric proof of the density formula for regular model sets.

\subsection{Definition and elementary properties}

Let $(G,H,\cL)$ be a \textit{cut-and-project scheme}, i.e., both $G$ and $H$ are unimodular locally compact groups, and $\cL$ is a uniform lattice in $G\times H$, which projects injectively to $G$ and densely to $H$. Let us denote the canonical projections by $\pi^G$ and $\pi^H$.   If $W\subseteq H$ is relatively compact, then the set $\Lambda_W=\pi^G(\cL\cap(G\times W))$  is called a \textit{weak model set}. Any weak model set is left-uniformly discrete and right-uniformly discrete. A \textit{model set} is a weak model set which satisfies $\mathring{W}\ne \varnothing$. Here and in the sequel of this paper, we denote by $\mathring{S}$ the interior of a set $S$ in a topological space.
  A model set is both right-relatively dense and left-relatively dense. We call a model set \textit{regular} if $W$ is Riemann measurable, i.e., if $m_H(\partial W)=0$, a nomenclature that is nowadays common in the abelian case \cite[Def.~7.2]{BG2}. Recall that the topological boundary of set $S$ in a topological spaces is defined as $\partial S:= \overline{S} \cap \overline{S^c}$. Note that in other contexts, the definition of regular model set can be slightly different.  For instance, in \cite{BHP17,BHP17II}, regular model sets are defined for non-abelian groups and homogeneous spaces arising from them  via cut-and-project schemes with a slightly more restrictive notion of window but also for possibly non-uniform lattices.
We refer to \cite{RVM3, Moody00, BG2} for further background on model sets in mostly abelian situations.

\smallskip

As $\pi^H(\cL)$ is dense in $H$, the lattice $\cL$ admits fundamental domains in $G\times H$ that are ``closely aligned to $G$''. This is formalised in the following lemma, which combines \cite[Lem.~2.3]{HR15} and Lemma~\ref{prop:exfd}.

\begin{lemma}\label{lem:smallfd}
Let $(G, H,\mathcal L)$ be a cut-and-project scheme and consider any non-empty, open and relatively compact set $U\subseteq H$. Then there exists an open relatively compact set $V\subseteq G$  satisfying
\begin{displaymath}
\mathcal L (V\times U) =G\times H.
\end{displaymath}
Moreover $V\times U$ contains a measurable relatively compact left-fundamental domain for $\cL$. \qed
\end{lemma}

\subsection{Leptin density of weak model sets}

The above property allows to derive bounds on the strong upper and lower densities of any weak model sets in the same way as in the lattice case.

\begin{theorem}[Leptin density estimates for weak model sets]\label{prop:df1}
Let  $(G,H,\cL)$ be a cut-and-project scheme, where $G$ is amenable.
Fix a relatively compact set $W\subseteq H$. We then have
\begin{displaymath}
 \frac{m_H(\mathring{W})}{\mathrm{covol}(\mathcal{L})} \le \Lep^-_{\Lambda_{\mathring{W}}} \le
 \Lep^-_{\Lambda_{W}} \le \Lep^+_{\Lambda_{W}}\le
 \Lep^+_{\Lambda_{\overline{W}}} \le
  \frac{m_H(\overline{W})}{\mathrm{covol}(\mathcal{L})} \ .
\end{displaymath}
In particular, if $W$ is Riemann measurable, then
\begin{displaymath}
\Lep_{\Lambda_{W}}=  \frac{m_H(W)}{\mathrm{covol}(\mathcal{L})}\ .
\end{displaymath}
\end{theorem}

\begin{remark} \label{rem:densityBHP}
	The so-called density formula $L_{\Lambda_W} =  m_H(W) / \mathrm{covol}(\mathcal{L})$ for Riemann measurable windows has been proven in the literature in various contexts, compare the discussion in the introduction. In \cite[Sec.~4.5]{BHP17} the notion of covolume of a regular model set  over locally compact second countable groups is defined. We emphasize that regular model sets as of \cite{BHP17, BHP17II} are defined via cut-and-project schemes with additional assumptions on $W$, but $\mathcal{L}$ does not need to be uniform, and the construction works beyond the realm of amenable groups.
	In the language of the present paper, the covolume of a regular model set is the inverse of its Leptin density. Sticking to the framework of \cite{BHP17}, a density approximation formula for regular model sets in amenable locally compact second countable groups can be derived by combining Prop.~4.13 with the ergodic theorem in Cor.~5.4 of \cite{BHP17}.
\end{remark}

\begin{proof}
(i) Consider any compact unit neighborhood $U\subseteq H$. By Lemma~\ref{lem:smallfd}, there exists a compact set $F$ such that $\cL (F \times U)=G \times H$, and we can pick a relatively compact measurable left-fundamental domain $\cF_U \subseteq F \times U$.  
Define $\cK_U=\cF_U^{-1}\cF_U$, write $K= \overline{F^{-1}F}$ and note that $\cK_U\subseteq K\times  U^{-1}U$.
For an arbitrary compact set $\cA \subseteq G \times H$, consider
\begin{displaymath}
\cL_\cA=\{\ell\in \cL: \cA\cap\cF_U^{-1} \ell\ne\varnothing\}\ , \qquad \cA_U=\bigcup_{\ell\in \cL_\cA} \cF_U^{-1}\ell \ .
\end{displaymath}
Then $\cL_\cA$ is a finite set and $\cA_U$ is a disjoint union.
We have
\begin{displaymath}
\cA\subseteq \cA_U\subseteq \cK_U\cA\ .
\end{displaymath}
For the second implication, assume that $x\in \cA_U$. Then $x\in\cF_U^{-1}\ell$ for some $\ell\in \cL_{\cA}$ and there exists some $a\in \cA$ such that $a\in \cF_U^{-1}\ell$. Hence $x\in \cF_U^{-1}\ell\subseteq \cF_U^{-1}\cF_Ua\subseteq \cF_U^{-1}\cF_U\cA\subseteq \cK_U\cA$.

\smallskip

\noindent (ii)  We first prove the last two upper inequalities. Fix  arbitrary $\varepsilon>0$.
Choose $U$ in (i) sufficiently small such that $m_H(U^{-1}U\overline{W})\le (1+\varepsilon)\cdot m_H(\overline W)$, which is possible due to continuity of the Haar measure on $H$. For arbitrary compact $A\subseteq G$  and $\cA = A \times \overline{W}$ we estimate with equality~\eqref{eqn:FD}
\begin{displaymath}
\begin{split}
m_{G\times H}(\cK_U\cA)&\ge m_{G\times H}(\cA_U)= m_{G\times H}(\cF_U)\cdot \mcard(\cL_\cA)=m_{G\times H}(\cF_U)\cdot \sum_{\ell\in \cL_\cA} \mcard(\cL\cap \cF_U^{-1}\ell)\\
&=m_{G\times H}(\cF_U)\cdot \mcard(\cL\cap \cA_U)\ge m_{G\times H}(\cF_U)\cdot \mcard(\cL\cap \cA) \ .
\end{split}
\end{displaymath}
Combining the latter estimate with $m_{G\times H}(\cK_U\cA)\le m_G(KA)\cdot m_H(U^{-1}U\overline{W})$, we arrive at
\begin{equation}\label{eq:rhs}
\begin{split}
(1+\varepsilon)\cdot \frac{m_H(\overline{W})}{m_{G\times H}(\cF_U)} &\cdot m_G(KA) \ge\frac{m_H(U^{-1}U\overline{W})}{m_{G\times H}(\cF_U)} \cdot m_G(KA) \ge \mcard(\cL\cap \cA)\\
& \ge \mcard(\Lambda_{\overline{W}}\cap A) \ge \mcard(\Lambda_{W}\cap A)  \ .
\end{split}
\end{equation}
Now the last two upper inequalities claimed in the theorem follow by definition of the upper Leptin density.

\smallskip

\noindent (iii) To show the first two lower inequalities, fix arbitrary $\varepsilon>0$.  Take  a compact set $V\subseteq \mathring{W}$ sufficiently large such that
$m_H(V)\ge (1-\varepsilon)\cdot m_H(\mathring{W})$, which is possible by regularity from below of the Haar measure.

\noindent
Choose a compact zero neighborhood $U\subseteq H$ such that $U^{-1}UV\subseteq \mathring{W}$, compare \cite[Lemma~4.1.3]{DE}. Consider arbitrary compact $A\subseteq G$ and write $\cA=A\times V$. Define $\cF_U$ and $\cK_U$ as in step~(i) above.
Using the equality~\eqref{eqn:FD}, we compute
\begin{displaymath}
\begin{split}
m_{G\times H}(\cA)&\le m_{G\times H}(\cA_U){=} m_{G\times H}(\cF_U)\cdot \mcard(\cL_\cA)=m_{G\times H}(\cF_U)\cdot \sum_{\ell\in \cL_\cA} \mcard(\cL\cap \cF_U^{-1}\ell)\\
&=m_{G\times H}(\cF_U)\cdot \mcard(\cL\cap \cA_U)\le m_{G\times H}(\cF_U)\cdot \mcard(\cL\cap \cK_U\cA) \ .
\end{split}
\end{displaymath}
Noting that $\cK_U\cA\subseteq KA\times U^{-1} U V\subseteq KA\times \mathring{W}$, we infer
\begin{displaymath}
\mcard(\cL\cap \cK_U\cA)\le \mcard(\Lambda_{\mathring{W}}\cap KA)
\le \mcard(\Lambda_{W}\cap KA) \ .
\end{displaymath}
Note that in the first inequality above we used the injectivity of the projection $\pi^G$.
Putting everything together yields
\begin{align} \label{eqn:rhs2}
& (1-\varepsilon)\cdot  m_H(\mathring{W}) \cdot m_G(A) \leq m_H(V) \cdot m_G(A)
 \leq m_{G \times H}(\cF_U) \cdot \mcard(\Lambda_{\mathring{W}} \cap KA) \nonumber \\
 & \quad \leq m_{G \times H}(\cF_U) \cdot \mcard(\Lambda_{{W}} \cap KA).
\end{align}
Now the first two lower inequalities claimed in the theorem follow by definition of the lower Leptin density.

\smallskip

\noindent (iv) Recall that $\Lep^-_{\Lambda_{W}}\le \Lep^+_{\Lambda_{W}}$ by amenability of $G$, see  Lemma~\ref{lem:Smp}.
Now the ``In particular''-part of the theorem follows from  $m_H(\mathring W)=m_H(W)=m_H(\overline{W})$ by Riemann measurability of $W$.
\end{proof}

\noindent
The inequalities~\eqref{eq:rhs} and~\eqref{eqn:rhs2} in the preceding proof lead to the following uniform estimates.

\begin{cor}
	Let $(G,H,\mathcal{L})$ be a cut-and-project scheme as in the previous theorem. Let $(A_i)_{i\in\mathbb I}$ be a strong F{\o}lner net in $G$. Then for all $\varepsilon> 0$ there is some $i_0 \in \mathbb{I}$ such that for all $i \succ i_0$ and all $(x,h) \in G \times H$, one gets the uniform estimates
	\[
	\frac{ m_H(\mathring{W})}{\mathrm{covol}(\mathcal{L})} - \varepsilon \le \frac{\mcard(\Lambda_{Wh} \cap A_ix)}{m_G(A_i)} \leq  \frac{m_H(\overline{W})}{\mathrm{covol}(\mathcal{L})} + \varepsilon.
	\]
	In particular, if $W$ is Riemann measurable then the convergence to the limit is uniform in $(x,h)$.
\end{cor}

\begin{proof}
	Let $\varepsilon > 0$.
 It follows from the inequality~\eqref{eq:rhs} with $W$ replaced by $Wh$, from the unimodularity of the group $G$ and from the F{\o}lner property of $(A_i)_{i\in \mathbb I}$ that there must be some $i_0 \in \mathbb{I}$ such that
 \[
 \sup_{(x,h) \in G \times H} \frac{\mcard(\Lambda_{Wh} \cap A_ix)}{m_G(A_i)} \leq (1+ \varepsilon)^2 \cdot \frac{m_H(\overline{W})}{\mathrm{covol}(\mathcal{L})}
 \]
 for all $i \succ i_0$. Adjusting $\varepsilon$ gives the claimed upper bound.

\smallskip

\noindent As for the lower bound, we consider the compact set $K$ in inequality~\eqref{eqn:rhs2} in the proof of Theorem~\ref{prop:df1} and claim that for given $\varepsilon>0$ there must be some $i_1 \succ i_0$ such that
 \[
 \inf_{(x,h) \in G \times H} \frac{\mcard(\Lambda_{Wh} \cap A_i x)}{m_G(A_i)} \ge
 \inf_{(x,h) \in G \times H} \frac{\mcard(\Lambda_{Wh} \cap KA_i x)}{m_G(A_i)} -\varepsilon
 \]
 for all $i \succ i_1$. To this end, fix a finite constant $C$ and a symmetric unit neighborhood $B \subseteq G \times H$ such that $\delta_{\cL}(B^2(x,h))\le C$. Using Lemma~\ref{lem:tbb} (which is justified as $G\times H$ is unimodular) and the comparison Lemma~\ref{lem:vHsF} (which is justified as $K$ is symmetric and contains the identity), we estimate
\begin{displaymath}
\begin{split}
 \mcard (\Lambda_{Wh} \cap (A x) &\, \Delta\,  (KAx))
 =  \mcard (\Lambda_{Wh} \cap \delta^{K}(A x)) \leq  \mcard (\Lambda_{Wh} \cap \partial_{K}(Ax)) \\
&\leq \frac{C}{m_{B\times H}(B)} \cdot   m_{G \times H}(\pi^G(B) \partial_{K}(Ax) \times \pi^H(B)Wh)\\
 &=  \frac{C}{m_{G\times H}(B)} \cdot  m_G(\pi^G(B)\partial_K A)\cdot m_H(\pi^H(B)W) \ .
\end{split}
 \end{displaymath}
Note that the above estimate is uniform in $(x,h)\in G\times H$. Thus
dividing by $m_G(A)$ and using the strong F{\o}lner property of Lemma~\ref{lem:assm} yields the claim.
But inequality~\eqref{eqn:rhs2} yields
 \[
 \inf_{(x,h) \in G \times H} \frac{\mcard(\Lambda_{Wh} \cap KA_i x)}{m_G(A_i)} \geq (1- \varepsilon) \cdot \frac{m_H(\mathring{W})}{\mathrm{covol}(\mathcal{L})}
 \]
 for all $i$. Now adjusting $\varepsilon$ finishes the proof.
\end{proof}

\section{Almost periodicity of regular model sets}

\subsection{Almost periodicity for point sets and measures}

With respect to mathematical diffraction theory, almost periodicity of the underlying point set has become a central notion in recent years.  For upper translation bounded measures in $\sigma$-compact locally compact abelian groups, this is analysed in \cite{LSS20}.  In particular, pure point diffraction of a measure is shown to be equivalent to so-called mean almost periodicity of the measure \cite[Thm.~2.13]{LSS20}, see also below. The discussion is based on density notions using van Hove sequences.  Here we will work with the stronger uniform version of mean almost periodicity, thereby avoiding averaging sequences or nets.

\begin{defi}
A model set $\Lambda$ in an amenable locally compact group $G$ is called {\em right-uniformly mean almost periodic} if for every $\varepsilon>0$ there exists a both left- and right-relatively dense set $T\subseteq G$ such that $\Lep^+_{\Lambda\,\triangle\,(\Lambda t)}\le \varepsilon$ for every $t\in T$.
\end{defi}
	
\begin{remark}
Left-uniformly mean almost periodic model sets can be defined analogously.
 Recall that groups containing a model set are necessarily unimodular. Due to invariance of the Leptin densities with respect to translations, see Lemma~\ref{sec:lepinv}, the set $T$ in the above definition can be chosen to be  symmetric, i.e., $T=T^{-1}$. In this situation, $T$ is right-relatively dense if and only it if left-relatively dense.
\end{remark}
	
Note that, for $\sigma$-compact locally compact abelian groups, the latter is a weaker notion than Weyl almost periodicity \cite[Def.~4.1]{LSS20}, the latter being defined via approximation by trigonometric polynomials. Also note that Meyer's and Guih\'eneuf's definition of almost periodic pattern \cite[Def.~7]{GM14}, \cite[Def.~4]{G18} and almost periodic measure \cite[Def.~10]{GM14} coincide with Weyl almost periodicity. Note finally that the notion of Weyl almost periodicity is based on the Weyl seminorm \cite[Sec.~1.3]{LSS20}, which can be rephrased using the Leptin upper density, without resorting to a van Hove sequence or net.

\subsection{Uniform mean almost periodicity and model sets}

We analyse (right-)uniform mean almost periodicity for regular model sets.

\begin{theorem}
Any regular model set in an amenable locally compact group is right-uniformly mean almost periodic.
\end{theorem}

This is an immediate consequence of the following result, which may be seen as weak variant of Lemma~3 from \cite{GM14}, however without resorting to an averaging net.

\begin{lemma}
Let $\Lambda_W$ be a regular model set. Then for every $\varepsilon>0$  there exists a compact symmetric unit neighborhood $U\subseteq H$ such that $\Lep^+_{\Lambda_{(WU) \cap (W^cU)}}\le \varepsilon$, and $t\in \Lambda_U$ implies $\Lambda_W\,\triangle\, (\Lambda_W t)\subseteq \Lambda_{(WU) \cap (W^c U)}$.
\end{lemma}

\begin{proof}
Take any compact symmetric unit neighborhood $U\subseteq H$. Consider $t\in \Lambda_U$, write $t=\pi^G(\ell)$ for $\ell\in \cL\cap (G\times U)$  and note
\begin{displaymath}
(\Lambda_W t) \, \triangle \, \Lambda_W = \Lambda_{(W \pi^H(\ell))} \, \triangle \, \Lambda_W \subseteq \Lambda_{(WU) \cap (W^cU)} \ .
\end{displaymath}
Now note that $m_H(\overline{WU \cap W^cU})\to m_H(\partial W)=0$ as $U\to \{e\}$ due to continuity of $m_H$. We thus have by Theorem~\ref{prop:df1} that $\Lep^+_{\Lambda_{WU \cap W^cU}}\le \mathrm{covol}(\mathcal{L})^{-1} m(\overline{WU \cap W^cU})\to 0$ as $U\to \{e\}$.
\end{proof}

To put this result into perspective, take any strong F{\o}lner net $(A_i)_{i\in\mathbb I}$ in $G$. We then have the estimate
\begin{displaymath}
d(t):=\limsup_{i\in \mathbb I} \frac{\text{card}((\Lambda_W t)\, \triangle \, \Lambda_W\cap A_i)}{m(A_i)} \leq  \Lep^+_{\Lambda_{(WU) \cap (W^cU)}} \ ,
\end{displaymath}
see Proposition~\ref{def:bdensities}. In particular, the set $P_\varepsilon(\Lambda_W)=\{ t\in G : d(t) \le \varepsilon \}$ is both right- and left-relatively dense for all $\varepsilon>0$. 

\medskip

In the context of $\sigma$-compact locally compact abelian groups, the property of $P_{\varepsilon}(\Lambda_W)$ being relatively dense for all $\varepsilon > 0$ is called mean almost periodicity of $\Lambda_W$, see \cite[Thm.~2.18]{LSS20}. Note that for $\Lambda$ being an arbitrary Meyer set, mean almost periodicity of $\Lambda$ is equivalent to  pure point diffraction of $\Lambda$. Moreover, any regular model set $\Lambda_W$ is a Meyer set and thus, $\Lambda_W$ has pure point diffraction.  Whereas this has been known since \cite[Thm.~5]{BM} and \cite[Thm.~1.1]{G05}, it has recently been put into broader perspective in \cite[Thms.~2.18, 2.13]{LSS20}.
We also point out that uniform mean almost periodicity of $\Lambda$ is a stronger notion than mean almost periodicity, that typically results in continuous dynamical eigenfunctions, compare \cite[Thm.~5]{Len}. In light of these results, it might also be interesting to study variants and aspects of uniformly mean almost periodic point sets beyond the abelian situation.

\appendix

\section{Nets and convergence}\label{sec:nets}

We briefly explain our definition of net and collect some basic properties. We adopt the setting of Moore-Smith convergence as described in \cite[Ch.~2]{K55}.

\subsection{Directed sets}

Let  $\mathbb I$ be a set, and let $\prec$ be a binary relation on $\mathbb I$. Recall the following properties that $(\mathbb I, \prec)$ might have.
\begin{itemize}
\item  reflexiveness: For all $i\in\mathbb I$ we have $i \prec i$.
\item anti-symmetry:  For all $i,j\in \mathbb I$, we have that $i\prec j$ and $j\prec i$ implies $i=j$.
\item  transitivity: For all $i,j,k\in \mathbb I$ we have that $i\prec j$ and $j\prec k$ implies $i\prec k$.
\item directedness: For all $i,j\in \mathbb I$ there exists $k\in \mathbb I$ such that $i\prec k$ and $j\prec k$.
\end{itemize}
We say that $(\mathbb I, \prec)$ is a \textit{partial order} if $\prec$ is reflexive, transitive and antisymmetric. In that case we call $\mathbb I$ a partially ordered set or poset. Examples of directed partial orders are $(\mathbb N, \le)$, $(\mathbb R, \le)$, $(\mathbb R, \ge)$, $(\mathcal P(X), \subseteq)$ and $(\mathcal P(X), \supseteq)$, where $\mathcal P(X)$ is the collection of subsets of the set $X$.

\subsection{Nets, cluster points and limits}

Let $(\mathbb I, \prec)$ a directed poset and let $X$ be a topological space. A \textit{net} in $X$ is a map $\mathbb I\to X$. We denote nets by $(x_i)_{i\in \mathbb I}$. The \textit{range of a net} is the set $\{x_i: i \in\mathbb I\}\subseteq X$. We call $x\in X$ a \textit{cluster point} of a net $(x_i)_{i\in \mathbb I}$ if for every neighborhood $U$ of $x$ and for every $i_0\in \mathbb I$ there exists $i\succ i_0$ such that $x_i\in U$.
A net $(x_i)_{i\in\mathbb I}$ \textit{converges} to $x\in X$ if for every neighborhood $U$ of $x$ there exists $i_0\in \mathbb I$ such that $x_i\in U$ for all $i\succ i_0$. In this case we say that $x$ is a \textit{limit point} of $
(x_i)_{i\in \mathbb I}$. Any limit point is a cluster point. Limit points are unique if and only if $X$ is a Hausdorff space \cite[Ch.~2, Th,~3]{K55}.

\subsection{Subnets and cluster points}

Let $X$ be a topological space and let $(\mathbb I, \prec)$ and $(\mathbb J, \prec)$ be directed preorders. Then a net $(y_j)_{j\in \mathbb J}$ in $X$ is a \textit{subnet} of $(x_i)_{i\in \mathbb I}$ in $X$ if $y_j=x_{\phi(j)}$ for some function $\phi: \mathbb J\to\mathbb I$ which is \textit{strictly cofinal}, i.e., for every $i_0\in\mathbb I$ there exists $j_0\in \mathbb J$ such that $j\succ j_0$ implies $\phi(j)\succ i_0$. Note that if $(x_i)_{i\in \mathbb I}$ converges to $x$, then every subnet $(y_j)_{j\in \mathbb J}$ converges to $x$. This is a direct consequence of the subnet definition. Also note that the composition of two cofinal maps is a cofinal map. Hence a subnet of a subnet is a subnet of the original net. The following characterisation of a cluster point is standard \cite[Ch.~2, Thm.~6]{K55}.

\begin{prop}\label{prop:charcp}
A point $x\in X$ is a cluster point of a net $(x_i)_{i\in \mathbb I}$ if and only if there exists a subnet $(x_{\phi(j)})_{j\in \mathbb J}$ that converges to $x$. \qed
\end{prop}

\subsection{Nets in $\overline{\mathbb R}$}

Let us  consider nets in the affinely extended real numbers $\overline{\mathbb R}=\mathbb R \cup \{\infty, - \infty\}$.
A net $(x_i)_{i\in \mathbb I}$ in $\overline{\mathbb R}$ is \textit{increasing} if $x_i \le  x_j$ for all $i\prec j$. It is  \textit{decreasing} if $x_i \ge  x_j$ for all $i\prec j$. As in the sequence case, we have the following result.

\begin{lemma}
Let $(x_i)_{i\in\mathbb I}$ be an increasing net in $\overline{\mathbb R}$. Then $(x_i)_{i\in \mathbb I}$ converges to its supremum $x=\sup\{x_i: i \in \mathbb I\}\le\infty$.  \qed
\end{lemma}

Let $(x_i)_{i\in\mathbb I}$ be a net in $\overline{\mathbb R}$. Then the net $(s_i)_{i\in\mathbb I}$ where $s_i=\inf\{x_j: j\succ i\}$ is increasing, and the net
 $(t_i)_{i\in\mathbb I}$ where $t_i=\sup\{x_j: j\succ i\}$ is decreasing. Note $s_i\le x_i\le t_i$ for all $i\in \mathbb I$ as $\prec$ is reflexive. As the nets $(s_i)_{i\in \mathbb I}$ and  $(t_i)_{i\in \mathbb I}$ both have a unique limit point, we can define
\begin{equation}\label{eq:limsup}
\begin{split}
\liminf_{i\in \mathbb I} x_i &= \lim_{i\in \mathbb I} \inf_{ j\succ i} x_j=\sup_{i\in \mathbb I}  \inf_{ j\succ i} x_j \ge -\infty \ ,
\qquad
\limsup_{i\in \mathbb I} x_i = \lim_{i\in \mathbb I} \sup_{ j\succ i}  x_j  = \inf_{i\in \mathbb I}  \sup_{ j\succ i}  x_j \le \infty \ .
\end{split}
\end{equation}
As in the sequence case it is seen that $\liminf_{i\in \mathbb I}x_i$ is the smallest cluster point of $(x_i)_{i\in \mathbb I}$. Moreover one may easily construct a subnet $(x_{\phi(n)})_{n\in\mathbb N}$ which is a sequence and converges to $\liminf_{i\in \mathbb I}x_i$. Analogous results hold for $\limsup_{i\in \mathbb I}x_i$.

\section{Box decompositions in $\sigma$-compact groups}

Some $\sigma$-compact locally compact groups admit  F{\o}lner sequences of hierarchically nested monotiles. Such sequences can be used for box decomposition arguments. We indicate in this section how to prove equality of uniform Leptin and Beurling densities via this method, as was done in \cite[Sec.~7]{GKS08} for $G= \mathbb{R}^d$ and $\nu = \delta_{\Lambda}$ being a Dirac comb over a uniformly discrete set~$\Lambda$.

\medskip

In a $\sigma$-compact group we call a compact set $A$ a {\em right monotile for $G$} if there is a countable set $T$ such that $A\cdot T = G$ and for each $s,t \in T$ with $s \neq t$, we have $\mathring{A}s \cap \mathring{A}t = \varnothing$, where as above, $\mathring{A}$ denotes the interior of the set $A$. Analogously, one  defines left monotiles for $G$. In the following we will exclusively deal with right monotiles and refer to those as {\em monotiles for $G$}.

\begin{defi} \label{defi:monotiles}
	We call a F{\o}lner sequence $(A_n)_{n\in \mathbb N}$ a {\em  monotile F{\o}lner sequence} if for each $n \in \N$, the set  $A_n$ a is a monotile for $G$. We say that a monotile F{\o}lner sequence is {\em topologically nested} if for all $n \in \N$ there is some $s_n \in G$ such that $A_n \subseteq \mathring{A}_{n+1} s_n$.
\end{defi}

Monotile F{\o}lner sequences can be found in many amenable groups. In \cite{Weiss01}, Weiss has shown that all countable linear amenable groups and all  residually finite amenable groups admit such sequences. In the abelian situation, Emerson used in \cite[Thm.~5]{Eme68} the structure theorem for  compactly generated LCA groups for finding monotile F{\o}lner sequences. By passing to a subsequence if necessary one can make sure that these sequences are topologically nested.
The construction of Emerson has been used in the context of mathematical quasicrystals, see for example the proof of Lem.~1.1 in \cite{Martin2}. 
Moreover, topologically nested monotile F{\o}lner sequences can be constructed explicitly for many homogeneous Lie groups such as the Heisenberg group.

\begin{example}
Let $G = H_3(\mathbb{R})$ be the $3$-dimensional Heisenberg group with its group multiplication given by
\[
(a,b,c) \cdot (x,y,z) = (a+x, b + y, c + z + ay).
\]
For each $n \in \mathbb{N}$, the set $F_n := [0, n)^2 \times [0, n^2)$ is a fundamental domain for the (uniform) lattice
\[
\Gamma_n := \{ (k,l,m): k,l \in n\Z, \, m \in n^2\Z \big\}.
\]
Now it is easy to see that by setting $A_n := \overline{F}_n$ for $n \in \N$, one obtains a monotile strong F{\o}lner sequence for $G$. Moreover, we have $A_n \cdot t_n \subseteq \mathring{A}_{n+1}$ when setting
$t_n = (\varepsilon, \varepsilon, \varepsilon)$, with $0 < \varepsilon < 1$.
Therefore, $(A_n)_{n\in \mathbb N}$ is also topologically nested.
\end{example}

For illustration, let us give a box decomposition proof of  $B^-_\cA\le \Lep^-$ when $\cA = (A_n)_{n\in \mathbb N}$ is a topologically nested monotile strong F{\o}lner sequence in a $\sigma$-compact amenable unimodular group.

\begin{proof}[Proof of $B^-_\cA\le \Lep^-$]
Let $\cA=(A_n)_{n\in\mathbb N}$ be a F{\o}lner sequence of monotiles as of Definition~\ref{defi:monotiles} and suppose that $\cA$ is also a strong F{\o}lner sequence.  We show
\begin{displaymath}
\Lep^-=\sup_{K\in \cK} \inf_{A\in \cK_p} \frac{\nu(KA)}{m(A)}
\ge B_\cA^-=\lim_{n\to\infty} \inf_{s\in G} \frac{\nu(A_ns)}{m(A_n)}\ .
\end{displaymath}
Setting $B_n := \mathring{A}_n =  A_n \setminus \partial A_n$ we note that we have
\begin{displaymath}
B_\cA^-=\lim_{n\to\infty} \inf_{s\in G} \frac{\nu(B_ns)}{m(A_n)}\ .
\end{displaymath}
Indeed, this follows from standard estimates, using $A_ns\setminus B_ns = \partial A_ns\subseteq \partial_U A_ns$ for any relatively compact unit neighborhood $U$, and using the F{\o}lner property together with
\begin{displaymath}
\lim_{n\to\infty} \sup_{s\in G} \frac{\nu(\partial_U A_ns)}{m(A_n)}=0 \ ,
\end{displaymath}
cf.~Lemma~\ref{lem:assm}. Fix arbitrary $\varepsilon>0$. Take $n\in\mathbb N$ large enough such that for all $s\in G$ we have
\begin{displaymath}
\frac{\nu(B_ns)}{m(A_n)} \ge B_\cA^--\varepsilon \ .
\end{displaymath}
Consider an arbitrary compact set $A\in\cK_p$. Since  $B_n t_n = \mathring{A}_n t_n \supseteq A_{n-1}$ for some $t_n \in G$, by assumption on a topologically nested monotile F{\o}lner sequence, we find a finite set $I$ and $s_i \in G$ for $i \in I$ such that
\begin{displaymath}
A\subseteq \dot{\bigcup_{i\in I}} B_ns_i \subseteq \bigcup_{i\in I} A_ns_i \subseteq KA,
\end{displaymath}
where $K={A_nA_n^{-1}}$, compare the argument for the inclusions~\eqref{eq:afd} in the proof of Proposition~\ref{prop:sudl}. Now we can estimate for $n$ large enough
\begin{displaymath}
\begin{split}
\nu(KA) &\ge \nu\left(\bigcup_{i\in I} B_ns_i\right)=\sum_{i\in I} \nu(B_ns_i)
\ge |I|\cdot m(A_n)\cdot(B_\cA^--\varepsilon) \\
&\ge m\left(\bigcup_{i\in I} A_ns_i\right)\cdot(B_\cA^--\varepsilon) \ge m(A) \cdot(B_\cA^--\varepsilon) \ .
\end{split}
\end{displaymath}
As $A$ was arbitrary, we infer
\begin{displaymath}
\sup_{K\in\cK} \inf_{A\in \cK_p} \frac{\nu(KA)}{m(A)} \ge B_\cA^--\varepsilon \ .
\end{displaymath}
As $\varepsilon>0$ was arbitrary, we get $\Lep^-\ge B_\cA^-$.
\end{proof}

\section{The densities of Gr\"ochenig, Kutyniok, Seip}

We review the density notion from \cite{GKS08}, which has inspired our definition of Leptin density.  Consider any locally finite measure $\nu$ on $G$.
For example, $\nu$ might be the counting measure on some right uniformly discrete point set in $G$.
Assume that $G$ admits a uniform lattice. Then densities of $\nu$ are defined via point counting with respect to $\Gamma$.

\begin{defi}\label{defGKS}
Let $G$ be a unimodular locally compact group. Assume that $\Gamma$ is a uniform lattice in $G$. For locally finite measures $\mu, \nu$ on $G$ we write $\mu\le \nu$ if for every $\varepsilon>0$ there exists $K\in \cK$ such that for all $A\in \cK$ we have $(1-\varepsilon) \cdot \mu(A)\le \nu(KA)$. We define
\begin{displaymath}
\cd_\nu^- = \sup\{ \alpha\in [0,\infty): \alpha \delta_\Gamma\leq  \nu\}\ , \qquad
\cd_\nu^+ = \inf\{ \alpha\in [0,\infty): \nu\leq  \alpha \delta_\Gamma\} \ .
\end{displaymath}
Here $\delta_\Gamma$ is the point counting measure associated to $\Gamma$ via $\delta_\Gamma(A)=\mcard(A\cap \Gamma)$.
\end{defi}

\begin{prop}
Let $G$ be a unimodular locally compact group which admits a uniform lattice $\Gamma$. Let $\nu$ be any locally finite measure on $G$.  We then have
\begin{displaymath}
\Lep_\Gamma\cdot \cd_\nu^- =\Lep^-_\nu \ , \qquad
\Lep_\Gamma\cdot \cd_\nu^+= \Lep^+_\nu \ ,
\end{displaymath}
where $\Lep_\Gamma$ denotes the Leptin density of the lattice $\Gamma$.
\end{prop}

\begin{proof}
We argue $\Lep_\Gamma\cdot \cd^- =\Lep^-$, the proof for the upper density is analogous.

\noindent (i) We collect some prerequisites from the proof of Proposition~\ref{prop:sudl}. Let $F$ be a measurable relatively compact left-fundamental domain for $\Gamma$ and note $\Lep_\Gamma\cdot m(F)=1$. Defining  $K=\overline{F^{-1} F}$, we  have for any $A\in \cK$ the estimates
\begin{displaymath}
m(F)\cdot\delta_\Gamma(A)\le m(KA) \ , \qquad m(F)\cdot\delta_\Gamma(KA)\ge m(A)\ .
\end{displaymath}
\noindent (ii) To show $\cd^- \le m(F) \cdot \Lep^-$, we assume without loss of generality $\cd^- >0$.
Consider any $0<d< \cd^- $. Then by Definition~\ref{defGKS} there exists $K'\in\cK$ such that for all $A\in\cK$ we have $d\,\delta_\Gamma(A)\le \nu(K'A)$. This implies
$\nu(K'KA) \ge d\,\delta_\Gamma(KA) \ge d \cdot m(A)/m(F)$. We thus have $d/m(F) \le \Lep^-$. As $0<d< \cd^- $ was arbitrary, we conclude $\cd^- \le m(F) \cdot \Lep^-$.

\noindent (iii) To show $m(F)\cdot \Lep^-\le \cd^- $, we assume without loss of generality $\Lep^-> 0$.
Take any $0<d<\Lep^-$. Then by definition of the lower Leptin density, cf.\@ Definition~\ref{def:lep}, there exists $K'\in\cK$ such that for all $A\in\cK$ we have $d\cdot m(KA)\le \nu(K'KA)$.
This implies $m(F)\cdot d \, \delta_\Gamma(A)\le d \cdot m(KA)\le \nu(K'KA)$. We thus obtain $m(F)\cdot d < \cd^- $. As $0<d<\Lep^-$ was arbitrary, we conclude $m(F)\cdot \Lep^-\le \cd^- $.
\end{proof}

\end{document}